\documentclass[12pt,a4paper]{article}
\usepackage{latexsym, amsmath, amsfonts, amscd, amssymb, verbatim, mathrsfs}
\def\N{I\!\!N}
\def\R{I\!\!R}
\def\oR{I\!\!\overline{R}}

\def\mA{\mathcal{A}}

\def\mN{\mathcal{N}}
\def\mC{\mathcal{C}}

\def\mG{\mathcal{G}}

\def\mQ{\mathcal{Q}}

\def\mJ{\mathcal{J}}

\def\mU{\mathcal{U}}
\def\mV{\mathcal{V}}
\def\mS{\mathcal{S}}

\def\be{\bar{e}}

\def\ou{\bar{u}}\def\wu{\widetilde{u}}\def\hu{\widehat{u}}
\def\ov{\bar{v}}

\def\oB{\bar{B}}

\def\cl{{\rm cl}}
\def\gph{{\rm gph\,}}

\def\dom{{\rm dom\,}}

\def\cone{{\rm cone}}
\def\argmin{{\rm argmin}}

\def\st{|\,}\def\bst{\big|\,}\def\Bst{\Big|\,}
\def\Limsup{\mathop{{\rm Lim\hspace{0.3mm}sup}}}

\def\liminf{\mathop{{\rm lim\hspace{0.3mm}inf}}}

\begin{document}

\newtheorem{Theorem}{Theorem}[section]
\newtheorem{Proposition}[Theorem]{Proposition}
\newtheorem{Remark}[Theorem]{Remark}
\newtheorem{Lemma}[Theorem]{Lemma}
\newtheorem{Corollary}[Theorem]{Corollary}
\newtheorem{Definition}[Theorem]{Definition}
\newtheorem{Example}[Theorem]{Example}
\newtheorem{CounterExample}[Theorem]{Counter-Example}
\renewcommand{\theequation}{\thesection.\arabic{equation}}
\normalsize

\title{Full Stability for Variational Nash Equilibriums of Parametric
Optimal Control Problems of PDEs\footnote{\textbf{Funding:} The
first author was supported by the Alexander von Humboldt Foundation,
Germany. The second author was partially supported by the German
Research Foundation (DFG) within the priority program ``Non-smooth
and Complementarity-based Distributed Parameter Systems: Simulation
and Hierarchical Optimization'' (SPP 1962) under grant number WA
3626/3-2.}}

\author{Nguyen Thanh Qui\footnote{Department of Mathematics, College of Natural Sciences,
        Can Tho University, Campus II, 3/2 Street, Can Tho, Vietnam (ntqui@ctu.edu.vn);
        Institut f\"{u}r Mathematik, Universit\"{a}t W\"{u}rzburg, Emil-Fischer-Str.~30, 97074 W\"{u}rzburg, Germany
        (thanhqui.nguyen@mathematik.uni-wuerzburg.de).}
        ~~and~~Daniel Wachsmuth\footnote{Institut f\"{u}r Mathematik, Universit\"{a}t W\"{u}rzburg,
        Emil-Fischer-Str.~30, 97074 W\"{u}rzburg, Germany
        (daniel.wachsmuth@mathematik.uni-wuerzburg.de).}}

\maketitle
\date{}

\noindent {\bf Abstract.} This paper investigates full stability
properties for \emph{variational Nash equilibriums} of a system of
parametric nonconvex optimal control problems governed by semilinear
elliptic partial differential equations. We first obtain some new
results on the existence of variational Nash equilibriums for the
system of original/parametric nonconvex optimal control problems.
Then we establish explicit characterizations of the Lipschitzian and
H\"olderian full stability of variational Nash equilibriums under
perturbations. These results deduce the equivalence between
variational Nash equilibriums and local Nash equilibriums in the
classical sense.



\section{Introduction}

It is well-known that the noncooperative game interactions of $m$
players, where every player finds to optimize individually under the
effects of each other's choices, is a valuable model for competitive
circumstances in economics and operations management. This model is
one of the most importance in the theory of equilibrium problems,
where solutions of the model will be called (Nash) equilibriums. If
each optimization problem in the equilibrium model is convex, the
equilibriums are understood in the classical sense. Otherwise, these
equilibriums will be defined via the variational sense, i.e., they
satisfy a first-order necessary optimality condition in the form of
a generalized equation (a variational system/inequality). Concerning
with the equilibrium model, standard questions related to existence,
uniqueness and stability of equilibriums are interesting and they
need to be addressed. Recently, Rockafellar \cite{Rock18VJM} has
studied the strongly stable local optimality and the strong metric
regularity of variational Nash equilibriums for an abstract
game-like framework of multi-agent optimization via tools of
variational analysis.

In this paper, we will consider an equilibrium model for multi-agent
optimization, where each agent is an optimal control problem
governed by a semilinear elliptic partial differential equation with
the cost functional being nonconvex. By applying techniques of the
optimal control theory and results in the perturbation theory of
maximal monotone and $m$-accretive operators we obtain some new
results on the existence of variational Nash equilibriums for the
original/parametric (nonconvex) equilibrium problem. In addition, we
establish criteria of Lipschitzian and H\"olderian full stability
for the equilibrium problem under (basic and tilt) perturbations by
means of the Mordukhovich's generalized differentiation. The full
stability of the parametric equilibrium problem deduces the local
uniqueness of a variational Nash equilibrium in question and ensures
the equivalence between variational Nash equilibriums and local Nash
equilibriums in the classical sense.

Recent results for generalized Nash equilibrium problems associated
to convex optimal control problems governed by partial differential
equations can be found in
\cite{DreGwi16JOTA,HinSur13PJO,HiSuKo15SIOPT,KKSW19SIOPT}. In
comparison with these works, the methods and techniques we apply to
deal with the nonconvex Nash equilibrium problems in this paper are
different and new (even applying for the convex setting).

The concept of full Lipschitzian stability for local minimizers was
introduced and studied by Levy, Poliquin, and Rockafellar
\cite{LePoRo00SIOPT} in the setting of finite dimensions. Then this
property together with the full H\"olderian stability were
investigated in the infinite dimensional setting by Mordukhovich and
Nghia in the work \cite{MorNgh14SIOPT}, where the notion of full
H\"olderian stability was also introduced in \cite{MorNgh14SIOPT}.
Many researchers have been interested in the full stability for
local minimizers; see, e.g.,
\cite{LePoRo00SIOPT,MorNgh14SIOPT,MoNgRo15MOR,MoOuSa14SIOPT,MoRoSa13SIOPT,MorSar15NA,QuiDWa19SICON}.
The notions of full stability are defined via basic parameters and
tilt ones, where the tilt parameters are related to the concept of
tilt stability introduced by Poliquin and Rockafellar in the paper
\cite{PoRo98SIOPT}; see, e.g.,
\cite{DruLew13SIOPT,EbeWen12NA,GfrMor15SIOPT,LewZha13SIOPT,MorRoc12SIOPT}
for more results on the tilt stability. Note that the concepts of
Lipschitzian and H\"olderian full stability were defined for local
minimizers of parametric optimization problems. Extensions of these
concepts for parametric variational systems (PVSs), Mordukhovich and
Nghia \cite{MorNgh16SIOPT} have studied local strong maximal
monotonicity of set-valued operators in Hilbert spaces with
applications to full Lipschitzian and H\"olderian stability for
solutions of PVSs, where both notions of Lipschitzian and
H\"olderian full stability for solutions of PVSs were also
introduced in \cite{MorNgh16SIOPT}. In addition, the authors in
\cite{MorNgh16SIOPT} established characterizations of these notions
for PVSs. Note that the notions of full stability for solutions of
PVSs are generalized the ones for local minimizers of optimization
problems, and these notions are equivalent in some special cases;
see the comments in \cite{MorNgh16SIOPT} for more details. To the
best of our knowledge, there were only a few applications of the
results of \cite{MorNgh16SIOPT}; see, e.g., \cite{MoNgPh18SVVA}.
Recently, the results of \cite{MorNgh16SIOPT} were applied to some
models that can be regarded as general PVSs in \cite{MoNgPh18SVVA}.
In particular, necessary conditions and sufficient conditions of
full stability for solutions to general parametric variational
inequalities (PVIs) were established in \cite{MoNgPh18SVVA}. Note
that the authors of \cite{MoNgPh18SVVA} considered the PVIs over
fixed constraint sets.

In order to investigate stability for variational Nash equilibriums
of our equilibrium model when the optimal control problems in the
equilibrium model undergone full perturbations, we have recognized
that the results on the Lipschitzian and H\"olderian full stability
for solutions of PVSs given in \cite{MorNgh16SIOPT} are suitable and
effective in applications for our parametric equilibrium problem.
For this reason, we will apply the results of \cite{MorNgh16SIOPT}
and the techniques of the optimal control theory to establish
characterizations of the Lipschitzian and H\"olderian full stability
of variational Nash equilibriums. A crucial role for our stability
results is that a quadratic form defined via the second-order
directional derivatives of the cost functionals is a Legendre form.
The latter fact will be also proved in this paper.

The rest of the paper is organized as follows. Preliminaries given
in Section~2 consist of the definitions for
classical/local/variational Nash equilibriums, standard assumptions
and auxiliary results for optimal control problems, and some
material from variational analysis. Section~3 is devoted to prove
results on existence of variational/classical Nash equilibriums
associated to many finite optimal control problems governed by a
semilinear elliptic partial differential equation. The results
provided in this section are new and they are among our main results
in this paper. In Section~4, we investigate the full stability of
variational Nash equilibriums to the parametric equilibrium problem.
We first prove that the quadratic form defined via the second-order
directional derivatives of the cost functionals is a Legendre form.
Then we establish necessary conditions and sufficient conditions
(resp., explicit characterizations) of the Lipschitzian and
H\"olderian full stability for variational Nash equilibriums with
respect to general nonempty bounded closed convex admissible control
sets (resp., admissible control sets of box constraint type). We
will also prove for the case of admissible control sets of box
constraint type that variational Nash equilibriums and local Nash
equilibriums are equivalent under the full stability condition of
variational Nash equilibriums. Finally, some concluding remarks will
be given in the last section.

\section{Preliminaries}
\setcounter{equation}{0}

\subsection{Variational Nash Equilibrium}

For each $k\in\{1,\ldots,m\}$, consider the optimal control problem
\begin{equation}\label{KthCPrlm}
\begin{cases}
    \mbox{Minimize}\quad J_k(u_1,\ldots,u_m)=\displaystyle\int_{\Omega}L_k(x,y_u(x))dx
                                            +\frac{1}{2}\int_\Omega\zeta_k(x)u_k(x)^2dx
    \vspace{2mm}\\
    \mbox{subject to}\quad u_k\in\mU^k_{ad}\subset L^2(\Omega),
\end{cases}
\end{equation}
where $\zeta_k\in L^\infty(\Omega)$ satisfies
$\zeta_k(x)\geq\zeta_{0k}>0$ for a.a. $x\in\Omega\subset\R^N$, and
$y_u$ is the weak solution associated to the control
$u=\sum_{i=1}^mB_iu_i$ of the Dirichlet problem
\begin{equation}\label{StateEq}
\begin{cases}
\begin{aligned}
    \mA y+f(x,y)&=\sum_{i=1}^mB_iu_i\ &&\mbox{in}\ \Omega\\
               y&=0                   &&\mbox{on}\ \Gamma,
\end{aligned}
\end{cases}
\end{equation}
where $B_i\in L^\infty(\Omega)$ for all $i\in\{1,\ldots,m\}$ and
$\mA$ denotes the second-order differential elliptic operator of the
form
\begin{equation}\label{SOdrEOper}
    \mA y(x)=-\sum_{i=1}^N\sum_{j=1}^N\partial_{x_j}\big(a_{ij}(x)\partial_{x_i}y(x)\big),
\end{equation}
and the admissible control set $\mU^k_{ad}$ in \eqref{KthCPrlm} is
nonempty, convex, closed and bounded in $L^2(\Omega)$.

\begin{Example}\rm
Let us provide a specific example for the admissible control sets
$\mU^k_{ad}$ in \eqref{KthCPrlm}, for $k\in\{1,\ldots,m\}$, that are
very frequently appearing in the applications as follows
\begin{equation}\label{ExACtrSet}
\begin{cases}
    \mU^k_{ad}=\big\{v\in L^2(\Omega)\bst\alpha_k(x)\leq v(x)\leq\beta_k(x)~\mbox{for a.a.}~x\in\Omega\big\}
    \vspace{1.5mm}\\
    \mbox{where}~\alpha_k,\beta_k\in L^\infty(\Omega)~\mbox{with}~\alpha_k(x)<\beta_k(x)~\mbox{for a.a.}~x\in\Omega.
\end{cases}
\end{equation}
\end{Example}

An element $\ou_k\in\mU^k_{ad}$ is said to be a
\emph{solution}/\emph{global minimum} of the control
problem~\eqref{KthCPrlm} if $J_k(\ou_k)\leq J_k(u_k)$ holds for all
$u_k\in\mU^k_{ad}$. We will say that $\ou_k$ is a \emph{local
solution}/\emph{local minimum} of the problem~\eqref{KthCPrlm} if
there exists a closed ball $\oB_\varepsilon(\ou_k)\subset
L^2(\Omega)$ with the center $\ou_k$ and the radius $\varepsilon>0$
such that $J_k(\ou_k)\leq J_k(u_k)$ holds for all
$u_k\in\mU^k_{ad}\cap\oB_\varepsilon(\ou_k)$. The local solution
$\ou_k$ is called \emph{strict} if we have $J_k(\ou_k)<J_k(u_k)$ for
all $u_k\in\mU^k_{ad}\cap\oB_\varepsilon(\ou_k)$ with
$u_k\neq\ou_k$.

\begin{Definition}\rm
A \emph{Nash equilibrium} (in the classical sense) associated to the
optimal control problems of the form \eqref{KthCPrlm} is the
following combination
\begin{equation}\label{ClVNEqui}
    (\ou_1,\ldots,\ou_m)\in\mU_{ad}^1\times\cdots\times\mU_{ad}^m
\end{equation}
such that
\begin{equation}\label{ClKthOpPr}
    \ou_k\in\underset{u_k\in\mU^k_{ad}}\argmin J_k(u_k,\ou_{-k}),\ \forall k=1,\ldots,m,
\end{equation}
where $\ou_k$ is a decision associated to the $k$-th player and
$\ou_{-k}$ stands for those decisions of all other players. We say
that $(\ou_1,\ldots,\ou_m)$ is a \emph{local Nash equilibrium} if
for each $k\in\{1,\ldots,m\}$ there exists a closed ball
$\oB_\varepsilon(\ou_k)\subset L^2(\Omega)$ with the center $\ou_k$
and the radius $\varepsilon>0$ such that
\begin{equation}\label{LcKthOpPr}
    \ou_k\in\underset{u_k\in\mU^k_{ad}\cap\oB_\varepsilon(\ou_k)}\argmin J_k(u_k,\ou_{-k}),\ \forall k=1,\ldots,m.
\end{equation}
\end{Definition}

\begin{Definition}\rm
A \emph{variational Nash equilibrium} associated to the optimal
control problems of the form \eqref{KthCPrlm} is the following
combination
\begin{equation}\label{VNEqui}
    (\ou_1,\ldots,\ou_m)\in\mU_{ad}^1\times\cdots\times\mU_{ad}^m
\end{equation}
such that
\begin{equation}\label{KthPVI}
    0\in\nabla_{u_k}J_k(\ou_k,\ou_{-k})+N(\ou_k;\mU_{ad}^k),\ \forall k=1,\ldots,m,
\end{equation}
where $N(u;\Theta)$ is the normal cone to the convex set $\Theta$ at
$u$ in the sense of convex analysis. In other words, a variational
Nash equilibrium
$(\ou_1,\ldots,\ou_m)\in\mU_{ad}^1\times\cdots\times\mU_{ad}^m$ is a
solution of the system of necessary optimality conditions
\eqref{KthPVI} associated to \eqref{ClKthOpPr}.
\end{Definition}

For each $k\in\{1,\ldots,m\}$, we define the basic parametric cost
functional
\begin{equation}\label{BaPaCFctn}
    \mJ_k(u_k,u_{-k},e_Y,e_k)=J_k(u_k+e_Y,u_{-k})+(e_{k,J},y_{u+e_Y})_{L^2(\Omega)}
\end{equation}
and consider the corresponding fully perturbed problem of the
control problem \eqref{KthCPrlm}--\eqref{SOdrEOper} as follows
\begin{equation}\label{PrKthPrlm}
\begin{cases}
    \mbox{Minimize}\quad\mJ_k(u_k,u_{-k},e_Y,e_k)-(u^*_k,u_k)_{L^2(\Omega)}
    \vspace{2mm}\\
    \mbox{subject to}\quad u_k\in\mU^k_{ad}(e_k)\subset L^2(\Omega),
\end{cases}
\end{equation}
where $y_{u+e_Y}$ is the weak solution associated to the control
$u=\sum_{i=1}^mB_iu_i$ of the perturbed Dirichlet problem
\begin{equation}\label{PrStateEq}
\begin{cases}
\begin{aligned}
    \mA y+f(x,y)&=\sum_{i=1}^mB_iu_i+e_Y\ &&\mbox{in}\ \Omega\\
               y&=0                         &&\mbox{on}\ \Gamma.
\end{aligned}
\end{cases}
\end{equation}
Here, $e_Y\in E_Y$ and $e_k=(e_{k,J},e_{k,\bullet})\in E_{k,J}\times
E_{k,\bullet}=E_k$, where $E_Y\subset L^2(\Omega)$, $E_{k,J}\subset
L^2(\Omega)$ and $E_{k,\bullet}$ are metric parametric spaces.

\begin{Example}\rm
When the admissible control sets $\mU^k_{ad}$ are given by
\eqref{ExACtrSet}, we can consider the corresponding parametric
admissible control sets $\mU^k_{ad}(e_k)$ of the perturbed problem
\eqref{PrKthPrlm} by
\begin{equation}\label{NshAdSetK}
    \mU^k_{ad}(e_k)=\big\{v\in L^2(\Omega)\bst\alpha_k(x)+e_{k,\alpha}(x)\leq v(x)
    \leq\beta_k(x)+e_{k,\beta}(x)~\mbox{for a.a.}~x\in\Omega\big\},
\end{equation}
where
$e_k=(e_{k,J},e_{k,\bullet})=(e_{k,J},e_{k,\alpha},e_{k,\beta})$
with $e_{k,\bullet}=(e_{k,\alpha},e_{k,\beta})\in
L^\infty(\Omega)\times L^\infty(\Omega)$. For this case, the metric
parametric space is given by
\begin{equation}
    E_Y\times E_k
    =E_Y\times E_{k,J}\times E_{k,\bullet}
    \subset L^2(\Omega)\times L^2(\Omega)\times L^\infty(\Omega)\times L^\infty(\Omega)
\end{equation}
with the induced metric defined via the norm
\begin{equation}
\begin{aligned}
    \|(e_Y,e_k)\|
    &=\|(e_Y,e_{k,J},e_{k,\alpha},e_{k,\beta})\|\\
    &=\|e_Y\|_{L^2(\Omega)}+\|e_{k,J}\|_{L^2(\Omega)}+
    \|e_{k,\alpha}\|_{L^\infty(\Omega)}+\|e_{k,\beta}\|_{L^\infty(\Omega)}
\end{aligned}
\end{equation}
of the product space $L^2(\Omega)\times L^2(\Omega)\times
L^\infty(\Omega)\times L^\infty(\Omega)$.
\end{Example}

Let us define
$\mathbf{L}^2(\Omega)=L^2(\Omega)^m=L^2(\Omega)\times\cdots\times
L^2(\Omega)$ and $\mathbf{E}=E_Y\times E_1\times\cdots\times E_m$
the product spaces endowed with the sum norms therein. In what
follows we use the notations for
$\mathbf{u}\in\mathbf{L}^2(\Omega)$,
$\mathbf{u}^*\in\mathbf{L}^2(\Omega)$, $\mathbf{e}\in\mathbf{E}$ by
setting
\begin{equation}\label{AbrNtatn}
\begin{aligned}
    \mathbf{u}&=(u_1,\ldots,u_m)\in\mathbf{L}^2(\Omega),\\
    \mathbf{u}^*&=(u^*_1,\ldots,u^*_m)\in\mathbf{L}^2(\Omega),\\
    \mathbf{e}&=(e_Y,e_1,\ldots,e_m)\in\mathbf{E},
\end{aligned}
\end{equation}
and denote
\begin{equation}\label{NshAdESet}
    \mU_{ad}(\mathbf{e})=\mU^1_{ad}(e_1)\times\cdots\times\mU^m_{ad}(e_m)\subset\mathbf{L}^2(\Omega),
\end{equation}
where $\mU^k_{ad}(e_k)$ is given by \eqref{PrKthPrlm} for every
$k=1,\ldots,m$.

\begin{Definition}\rm
Given
$\mathbf{\ou}^*=(\ou^*_1,\ldots,\ou^*_m)\in\mathbf{L}^2(\Omega)$ and
$\mathbf{\be}=(\be_Y,\be_1,\ldots,\be_m)\in\mathbf{E}$, associated
to the parametric control problem \eqref{PrKthPrlm}, the following
combination
\begin{equation}\label{PerVNEqui}
    (\ou_1,\ldots,\ou_m)\in\mU_{ad}^1(\be_1)\times\cdots\times\mU_{ad}^m(\be_m)
\end{equation}
such that
\begin{equation}\label{PerKthPVI}
    \ou^*_k\in\nabla_{u_k}\mJ_k(\ou_k,\ou_{-k},\be_Y,\be_k)+N(\ou_k;\mU_{ad}^k(\be_k)),\ \forall k=1,\ldots,m,
\end{equation}
is called a \emph{variational Nash equilibrium with respect to the
parameters} $(\mathbf{\ou}^*,\mathbf{\be})$.
\end{Definition}

\subsection{Assumptions and auxiliary results}

Let us give some standard assumptions in optimal control and provide
some auxiliary results related to the optimal control problems
\eqref{KthCPrlm}--\eqref{SOdrEOper}.

We assume that $\Omega\subset\R^N$ with $N\in\{1,2,3\}$. In
addition, for every $k\in\{1,\ldots,m\}$, the functions
$L_k,f:\Omega\times\R\to\R$ are Carath\'eodory functions of the
class $\mC^2$ with respect to the second variable. We now consider
the following assumptions:

\textbf{(A1)} The function $f$ satisfies
$$f(\cdot,0)\in L^2(\Omega)\quad\mbox{and}\quad\dfrac{\partial f}{\partial y}(x,y)\geq0\quad
  \mbox{for a.a.}\ x\in\Omega,$$
and for all $M>0$ there exists a constant $C_{f,M}>0$ such that
$$\left|\dfrac{\partial f}{\partial y}(x,y)\right|+\left|\dfrac{\partial^2f}{\partial y^2}(x,y)\right|\leq C_{f,M},$$
and
$$\left|\dfrac{\partial^2f}{\partial y^2}(x,y_2)-\dfrac{\partial^2f}{\partial y^2}(x,y_1)\right|
  \leq C_{f,M}|y_2-y_1|,$$
for a.a. $x\in\Omega$ and $|y|,|y_1|,|y_2|\leq M$.

\textbf{(A2)} For every $k\in\{1,\ldots,m\}$, the function
$L_k(\cdot,0)\in L^1(\Omega)$, and for all $M>0$ there are a
constant $C_{L_k,M}>0$ and a function $\psi_{k,M}\in L^2(\Omega)$
such that
$$\left|\dfrac{\partial L_k}{\partial y}(x,y)\right|\leq\psi_{k,M}(x),\quad
  \left|\dfrac{\partial^2L_k}{\partial y^2}(x,y)\right|\leq C_{L_k,M},$$
and
$$\left|\dfrac{\partial^2L_k}{\partial y^2}(x,y_2)-\dfrac{\partial^2L_k}{\partial y^2}(x,y_1)\right|
  \leq C_{L_k,M}|y_2-y_1|,$$
for a.a. $x\in\Omega$ and $|y|,|y_1|,|y_2|\leq M$.

\textbf{(A3)} The set $\Omega$ is an open and bounded domain in
$\R^N$ with Lipschitz boundary $\Gamma$, the coefficients $a_{ij}\in
L^\infty(\Omega)$ of the second-order elliptic differential operator
$\mA$ defined by \eqref{SOdrEOper} satisfy the condition
$$\lambda_{\mA}\|\xi\|^2_{\R^N}\leq\sum_{i,j=1}^Na_{ij}(x)\xi_i\xi_j,\
  \forall\xi\in\R^N,\ \mbox{for a.a.}\ x\in\Omega,$$
for some constant $\lambda_{\mA}>0$.

For the sake of convenience in order to investigate properties of
weak solutions to the state equation \eqref{StateEq} which depend on
 controls given on the right-hand side of \eqref{StateEq}, we
consider the following auxiliary state equation
\begin{equation}\label{AuxStaEq}
\begin{cases}
\begin{aligned}
    \mA y+f(x,y)&=u     &&\mbox{in}\ \Omega\\
               y&=0     &&\mbox{on}\ \Gamma.
\end{aligned}
\end{cases}
\end{equation}
Then properties of weak solutions to the state equation
\eqref{StateEq} will be deduced from properties of weak solutions to
the equation \eqref{AuxStaEq}.

\begin{Theorem}\label{ThmExSoStEq}
Assume that the assumptions {\rm\textbf{(A1)}--\textbf{(A3)}} hold.
Then, for each $u\in L^2(\Omega)$, the state
equation~\eqref{AuxStaEq} admits a unique weak solution $y_u\in
H^1_0(\Omega)\cap C(\bar\Omega)$. In addition, for every
$k\in\{1,\ldots,m\}$, there exists a constant $M_k>0$ such that
\begin{equation}\label{EstSolEqSt}
    \|y_u\|_{H^1_0(\Omega)}+\|y_u\|_{C(\bar\Omega)}\leq M_k,\ \forall u\in\mU^k_{ad}.
\end{equation}
Furthermore, if $u^n\rightharpoonup u$ weakly in $L^2(\Omega)$, then
$y_{u^n}\to y_u$ strongly in $H^1_0(\Omega)\cap C(\bar\Omega)$.
\end{Theorem}
\begin{proof}
By applying \cite[Theorem~2.1]{CaReTr08SIOPT}, we obtain the
assertions of the theorem. $\hfill\Box$
\end{proof}

\begin{Theorem}\label{Thm24CRT}
Assume that the assumptions {\rm\textbf{(A1)}--\textbf{(A3)}} hold.
Then, associated to \eqref{AuxStaEq}, the control-to-state operator
$G:L^2(\Omega)\to H^1_0(\Omega)\cap C(\bar\Omega)$ defined by
$G(u)=y_u$ is of class $\mC^2$. Moreover, for every $u,v\in
L^2(\Omega)$, $z_{u,v}=G'(u)v$ is the unique weak solution of
\begin{equation}\label{EqSolZuv}
\begin{cases}
  \begin{aligned}
     \mA z+\frac{\partial f}{\partial y}(x,y)z&=v\ &&\mbox{in}\ \Omega\\
                                             z&=0  &&\mbox{on}\ \Gamma.
  \end{aligned}
\end{cases}
\end{equation}
Finally, for every $v_1,v_2\in L^2(\Omega)$,
$z_{v_1v_2}=G''(u)(v_1,v_2)$ is the unique weak solution of
\begin{equation}\label{EqSolSeGvv}
\begin{cases}
  \begin{aligned}
     \mA z+\frac{\partial f}{\partial y}(x,y)z+\frac{\partial^2f}{\partial y^2}(x,y)z_{u,v_1}z_{u,v_2}
      &=0\ &&\mbox{in}\ \Omega\\
     z&=0  &&\mbox{on}\ \Gamma,
  \end{aligned}
\end{cases}
\end{equation}
where $y=G(u)$ and $z_{u,v_i}=G'(u)v_i$ for $i=1,2$.
\end{Theorem}
\begin{proof}
The assertions of the theorem are deduced from
\cite[Theorem~2.4]{CaReTr08SIOPT}. $\hfill\Box$
\end{proof}

\medskip
More details related to weak solutions of the state equations
\eqref{AuxStaEq} as well as \eqref{StateEq} can be found in
\cite[Chapter~4]{Trolt10B}. By Theorem~\ref{Thm24CRT}, we denote the
space containing weak solutions of \eqref{AuxStaEq} by
$Y:=H^1_0(\Omega)\cap C(\bar\Omega)$ endowed with the norm
$$\|y\|_Y=\|y\|_{H^1_0(\Omega)}+\|y\|_{L^\infty(\Omega)}.$$
The forthcoming theorem provide us with formulas for computing the
first and second-order directional derivatives of the cost
functional of the control problem \eqref{KthCPrlm}.

\begin{Theorem}\label{ThmFSdDerJ}
Assume that the assumptions {\rm\textbf{(A1)}--\textbf{(A3)}} hold.
For every $k\in\{1,\ldots,m\}$, the cost functional
$J_k:\mathbf{L}^2(\Omega)\to\R$ is of class $\mC^2$. Moreover, for
every $(u_1,\ldots,u_m)\in\mathbf{L}^2(\Omega)$ and
$h,h_1,\ldots,h_m\in L^2(\Omega)$, the first and second derivatives
of $J_k(\cdot)$ are given by
\begin{equation}\label{FOdrDrtv}
    \nabla_{u_k}J_k(u_k,u_{-k})h=\int_\Omega(\zeta_ku_k+\varphi_{k,u})hdx,
\end{equation}
and
\begin{equation}\label{}
\begin{aligned}
    \nabla^2_{u_ku_j}J_k(u_1,\ldots,u_m)(h_k,h_j)
    &=\int_\Omega\left(\dfrac{\partial^2L_k}{\partial y^2}(x,y_u)
     -\varphi_{k,u}\dfrac{\partial^2f}{\partial y^2}(x,y_u)\right)z_{u,h_k}z_{u,h_j}dx\\
    &+\int_\Omega\chi_{\{k\}}(j)\zeta_kh_kh_jdx,
\end{aligned}
\end{equation}
where $u=\sum_{i=1}^mB_iu_i$, $y_u=G(u)$, $z_{u,h_i}=G'(u)h_i$ for
$i\in\{k,j\}$, and $\varphi_{k,u}\in W^{2,2}(\Omega)$ is the adjoint
state of $y_u$ defined as the unique  weak solution of
$$\begin{cases}
\begin{aligned}
    \mA^*\varphi+\dfrac{\partial f}{\partial y}(x,y_u)\varphi
                     &=\dfrac{\partial L_k}{\partial y}(x,y_u)\ &&\mbox{in}\ \Omega\\
             \varphi &=0                                        &&\mbox{on}\ \Gamma
\end{aligned}
\end{cases}$$
with $\mA^*$ being the adjoint operator of $\mA$.
\end{Theorem}
\begin{proof}
It follows from \cite[Theorem~2.6 and Remark~2.8]{CaReTr08SIOPT};
see also \cite[Chapter~4]{Trolt10B}. $\hfill\Box$
\end{proof}

\begin{Theorem}\label{ThmExSlOP}
Assume that the assumptions~{\rm\textbf{(A1)}--\textbf{(A3)}} hold.
For every $k\in\{1,\ldots,m\}$, the optimal control problem
\eqref{KthCPrlm} has at least one solution.
\end{Theorem}
\begin{proof}
Arguing similarly to the proof of \cite[Theorem~4.15]{Trolt10B}, we
obtain the assertion of the theorem; see also
\cite[Theorem~2.2]{CaReTr08SIOPT} for the related result.
$\hfill\Box$
\end{proof}

\subsection{Material from variational analysis}

This subsection recalls some concepts and facts of variational
analysis taken from \cite{Mor06Ba}; see also \cite{RoWe98B}. Unless
otherwise stated, every reference norm in a product normed space is
the sum norm. Let us denote the open ball of center $u\in X$ and
radius $\rho>0$ in a Banach space $X$ by $B_\rho(u)$, and
$\oB_\rho(u)$ is the corresponding closed ball. Let
$F:X\rightrightarrows W$ be a multifunction between Banach spaces.
The set $\gph F:=\{(u,v)\in X\times W\st v\in F(u)\}$ is the graph
of $F$. We say that $F$ is locally closed around
$\bar\omega:=(\ou,\ov)\in\gph F$ if the graph of $F$ is locally
closed around $\bar\omega$, i.e., there is a closed ball
$\oB_\rho(\bar\omega)$ such that $\oB_\rho(\bar\omega)\cap\gph F$ is
closed in $X\times W$. The \emph{sequential Painlev\'e-Kuratowski
outer/upper limit} of $\Phi:X\rightrightarrows X^*$ as $u\to \ou$ is
defined by
\begin{equation}\label{OtrLimit}
\begin{aligned}
    \displaystyle\Limsup_{u\to\ou}\Phi(u)=\Big\{
    &u^*\in X^*\Bst\mbox{there exist}~u_n\to\ou~\mbox{and}~u^*_n\stackrel{w^*}\rightharpoonup u^*~\mbox{with}\\
    &u^*_n\in\Phi(u_n)~\mbox{for every}~n\in\N:=\{1,2,\dots\}\Big\}.
\end{aligned}
\end{equation}
Let $\phi:X\to\oR$ be a proper extended-real-valued function on an
\emph{Asplund} space $X$ \cite{Asp68AM}; see
\cite{Mor06Ba,Mor06Bb,Phelp93B} for more applications of Asplund
spaces. Assume that $\phi$ is lower semicontinuous around $\ou$,
where $\ou\in\dom\phi:=\{u\in X\st\phi(u)<\infty\}$. The
\emph{regular subdifferential} of $\phi$ at $\ou\in\dom\phi$ is
\begin{equation}\label{RegSubdif}
    \widehat{\partial}\phi(\ou)=\bigg\{u^*\in X^*\Bst\liminf_{u\to\ou}
    \frac{\phi(u)-\phi(\ou)-\langle u^*,u-\ou\rangle}{\|u-\ou\|}\geq0\bigg\}.
\end{equation}
This concept is also known as the \emph{subdifferential in the sense
of viscosity solutions}; see \cite{CrLi83TAMS,CrEvLi84TAMS}. The
\emph{limiting subdifferential} (known also as \emph{Mordukhovich
subdifferential}) of $\phi$ at $\ou$ is defined via the
Painlev\'e-Kuratowski sequential outer limit \eqref{OtrLimit} by
\begin{equation}\label{LimSubdif}
    \partial\phi(\ou)=\Limsup_{u\stackrel{\phi}\to\ou}\widehat{\partial}\phi(u),
\end{equation}
where the notation $u\stackrel{\phi}\to\ou$ means that $u\to\ou$
with $\phi(u)\to\phi(\ou)$.

For a subset $\Theta\subset X$ locally closed around $\ou\in\Theta$,
the \emph{regular and limiting normal cones} to $\Theta$ at
$\ou\in\Theta$ are respectively defined by
\begin{equation}\label{RgLmtgNrCn}
    \widehat{N}(\ou;\Theta)=\widehat{\partial}\delta(\ou;\Theta)
    \quad\mbox{and}\quad N(\ou;\Theta)=\partial\delta(\ou;\Theta),
\end{equation}
where $\delta(\cdot;\Theta)$ is the indicator function of $\Theta$
defined by $\delta(u;\Theta)=0$ for $u\in\Theta$ and
$\delta(u;\Theta)=\infty$ otherwise. The \emph{regular} and
\emph{Mordukhovich coderivatives} of the multifunction
$F:X\rightrightarrows W$ at the point $(\ou,\ov)\in\gph F$ are
respectively the multifunction
$\widehat{D}^*F(\ou,\ov):W^*\rightrightarrows X^*$ defined by
$$\widehat{D}^*F(\ou,\ov)(v^*)=\big\{u^*\in X^*\bst(u^*,-v^*)\in\widehat{N}\big((\ou,\ov);\gph F\big)\big\},~\forall
  v^*\in W^*,$$
and the multifunction $D^*F(\ou,\ov):W^*\rightrightarrows X^*$ given
by
$$D^*F(\ou,\ov)(v^*)=\big\{u^*\in X^*\bst(u^*,-v^*)\in N\big((\ou,\ov);\gph F\big)\big\},~\forall v^*\in W^*.$$
Given any $\ou^*\in\partial\phi(\ou)$, following
\cite{MorNgh14SIOPT,MorNgh13NA} the \emph{combined second-order
subdifferential} of $\phi$ at $\ou$ relative to $\ou^*$ is the
multifunction
$\breve{\partial}^2\phi(\ou,\ou^*):X^{**}\rightrightarrows X^*$ with
the values
\begin{equation}\label{CSOdrSbdif}
    \breve{\partial}^2\phi(\ou,\ou^*)(u)=(\widehat{D}^*\partial\phi)(\ou,\ou^*)(u),~\forall u\in X^{**}.
\end{equation}
Note that for $\phi\in\mC^2$ around $\ou$ with
$\ou^*=\nabla\phi(\ou)$ we have
$\breve{\partial}^2\phi(\ou,\ou^*)(u)=\{\nabla^2\phi(\ou)u\}$ for
all $u\in X^{**}$ via the symmetric Hessian operator
$\nabla^2\phi(\ou)$.

We say that the multifunction $F:X\rightrightarrows W$ is
\emph{locally Lipschitz-like}, or $F$ has the \emph{Aubin property}
\cite{DontRock09B}, around a point $(\ou,\ov)\in\gph F$ if there
exist $\ell>0$ and neighborhoods $U$ of $\ou$, $V$ of $\ov$ such
that
$$F(u_1)\cap V\subset F(u_2)+\ell\|u_1-u_2\|\oB_W,~\forall u_1,u_2\in U,$$
where $\oB_W$ denotes the closed unit ball in $W$. Characterization
of this property via the mixed Mordukhovich coderivative of $F$ can
be found in \cite[Theorem~4.10]{Mor06Ba}.

\section{Existence of variational Nash equilibriums}
\setcounter{equation}{0}

In this section, under the standard assumptions
{\rm\textbf{(A1)}--\textbf{(A3)}} we will prove the existence of
variational Nash equilibriums to the system \eqref{KthPVI}
associated to the optimal control problems given in
\eqref{KthCPrlm}. This result will lead to the existence of
variational Nash equilibriums to the parametric system
\eqref{PerKthPVI} associated to the optimal control problems given
in \eqref{PrKthPrlm}. These variational Nash equilibriums are also
local Nash equilibriums in the classical sense when the optimal
control problems in question are all convex.

\begin{Theorem}\label{ThmExNshEq}
Assume that the assumptions {\rm\textbf{(A1)}--\textbf{(A3)}} hold.
There exists a variational Nash equilibrium
\begin{equation}\label{}
    \mathbf{\ou}=(\ou_1,\ldots,\ou_m)\in\mU_{ad}^1\times\cdots\times\mU_{ad}^m
\end{equation}
to the system
\begin{equation}\label{K1mthPVI}
    0\in\nabla_{u_k}J_k(\ou_k,\ou_{-k})+N(\ou_k;\mU_{ad}^k),\ \forall k=1,\ldots,m.
\end{equation}
\end{Theorem}
\begin{proof}
Let us put $\mathbf{X}=\mathbf{L}^2(\Omega)$ and put
$\zeta=(\zeta_1,\ldots,\zeta_m)$. For every
$\mathbf{u}=(u_1,\ldots,u_m)\in\mathbf{X}$, we denote
$\zeta\odot\mathbf{u}=(\zeta_1u_1,\ldots,\zeta_mu_m)$ and
$\varphi_{\mathbf{u}}=(\varphi_{1,u},\ldots,\varphi_{m,u})$, where
$u=\sum_{i=1}^mB_iu_i$ and $\varphi_{k,u}$ is the adjoint state of
$y_u$ with respect to the $k$-th control problem in
\eqref{KthCPrlm}. We now define the map
\begin{equation}\label{}
\begin{aligned}
    T:\mathbf{X}&\rightrightarrows\mathbf{X}\\
      \mathbf{u}&\mapsto\zeta\odot\mathbf{u}+N(\mathbf{u};\mU_{ad}).
\end{aligned}
\end{equation}
Denote $D(T)=\{\mathbf{u}\in\mathbf{X}\st
T(\mathbf{u})\neq\emptyset\}$ and choose the set
$\mG\equiv\mathbf{X}$. Then, $D(T)\equiv\mU_{ad}$ and the operator
$C$ defined by
\begin{equation}\label{}
\begin{aligned}
    C:\cl\,\mG\equiv\mathbf{X}&\to\mathbf{X}\\
      \mathbf{u}&\mapsto\varphi_{\mathbf{u}}
\end{aligned}
\end{equation}
is compact. Since $C$ is a compact operator, $C$ is completely
continuous (i.e., $C$ maps weakly convergent sequences into strongly
convergent sequences). In addition, we observe that $T$ is a maximal
monotone operator. Hence, $T$ is accretive and $R(T+\lambda
I)=\mathbf{X}$ for every $\lambda>0$ by
\cite[Corollary~3.7]{Phelp97EM}, where $I$ denotes the identity
operator on $\mathbf{X}$ which can be regarded as a duality mapping
and $R(T+\lambda I)$ stands for the range of $T+\lambda I$. Thus,
$T$ is $m$-accretive; see definitions of accretive/$m$-accretive
properties in \cite{Kar96TAMS}. Since $\mG\equiv\mathbf{X}$, the
conditions
\begin{equation}\label{}
    D(T)\cap\mG\neq\emptyset\quad\mbox{and}\quad\big(I-(T+C)\big)\big(D(T)\cap\partial\mG\big)\subset\cl\,\mG
\end{equation}
are trivial. Of course, $\mathbf{X}$ is a uniformly convex space.
Therefore, summarizing the above and applying
\cite[Theorem~1]{Kar96TAMS} we deduce that $0\in R(T+C)$, i.e.,
there exists $\mathbf{\ou}\in\mathbf{X}$ such that
\begin{equation}\label{}
    0\in T(\mathbf{\ou})+C(\mathbf{\ou}),
\end{equation}
or, equivalently, as follows
\begin{equation}\label{}
    0\in\zeta\odot\mathbf{\ou}+\varphi_{\mathbf{\ou}}+N(\mathbf{\ou};\mU_{ad}),
\end{equation}
which yields \eqref{K1mthPVI} due to
$\nabla_{u_k}J_k(\ou_k,\ou_{-k})=\zeta_k\ou_k+\varphi_{k,\ou}$ for
every $k=1,\ldots,m$. $\hfill\Box$
\end{proof}

\medskip
Based on Theorem~\ref{ThmExNshEq}, the forthcoming theorem provides
us with an existence result of local Nash equilibriums in the
classical sense to the equilibrium problems associated to the
optimal control problems \eqref{KthCPrlm} for the convex setting.

\begin{Theorem}\label{ThmCvNshEq}
Assume that the assumptions {\rm\textbf{(A1)}--\textbf{(A3)}} hold
and that the optimal control problems \eqref{KthCPrlm} are all
convex optimization problems. There exists a classical Nash
equilibrium
\begin{equation}\label{}
    \mathbf{\ou}=(\ou_1,\ldots,\ou_m)\in\mU_{ad}^1\times\cdots\times\mU_{ad}^m
\end{equation}
to the equilibrium problem associated to the optimal control
problems \eqref{KthCPrlm}.
\end{Theorem}
\begin{proof}
According to Theorem~\ref{ThmExNshEq}, there exists a variational
Nash equilibrium
$$\mathbf{\ou}=(\ou_1,\ldots,\ou_m)\in\mU_{ad}^1\times\cdots\times\mU_{ad}^m$$
satisfying the system \eqref{K1mthPVI}. Since the control problems
\eqref{KthCPrlm} are all convex, the system of conditions
\eqref{K1mthPVI} implies that
\begin{equation}\label{}
    \ou_k\in\underset{u_k\in\mU^k_{ad}}\argmin J_k(u_k,\ou_{-k}),\
    \forall k=1,\ldots,m,
\end{equation}
This means that $\mathbf{\ou}$ is a classical Nash equilibrium
associated to the problems \eqref{KthCPrlm}. $\hfill\Box$
\end{proof}

\medskip
By applying the techniques given in the proof of
Theorem~\ref{ThmExNshEq}, we can deduce the existence of variational
Nash equilibriums to the  parametric system \eqref{PerKthPVI}.

\begin{Theorem}\label{ThmPaNshEq}
Let $\mathbf{\ou}^*=(\ou^*_1,\ldots,\ou^*_m)\in\mathbf{L}^2(\Omega)$
and let $\mathbf{\be}=(\be_Y,\be_1,\ldots,\be_m)\in\mathbf{E}$ be
such that
\begin{equation}\label{CdUadNEmty}
    \mU_{ad}(\mathbf{\be})=\mU_{ad}^1(\be_1)\times\cdots\times\mU_{ad}^m(\be_m)\neq\emptyset.
\end{equation}
Assume that the assumptions {\rm\textbf{(A1)}--\textbf{(A3)}} hold.
There exists a variational Nash equilibrium
\begin{equation}\label{}
    \mathbf{\ou}=(\ou_1,\ldots,\ou_m)\in\mU_{ad}^1(\be_1)\times\cdots\times\mU_{ad}^m(\be_m)
\end{equation}
to the parametric system
\begin{equation}\label{PaKthPVI}
    \ou^*_k\in\nabla_{u_k}\mJ_k(\ou_k,\ou_{-k},\be_Y,\be_k)+N(\ou_k;\mU_{ad}^k(\be_k)),\
    \forall k=1,\ldots,m,
\end{equation}
with respect to the parameters $(\mathbf{\ou}^*,\mathbf{\be})$.
\end{Theorem}
\begin{proof}
Let $\mathbf{X}=\mathbf{L}^2(\Omega)$ and let
$\zeta=(\zeta_1,\ldots,\zeta_m)$. For every
$\mathbf{u}=(u_1,\ldots,u_m)\in\mathbf{X}$, we denote
\begin{equation}\label{}
    \varphi_{\mathbf{u},\mathbf{\be}}=(\varphi_{1,u+\be_Y,\be_1},\ldots,\varphi_{m,u+\be_Y,\be_m}),
\end{equation}
where $u+\be_Y=\sum_{i=1}^mB_iu_i+\be_Y$, the state $y_{u+\be_Y}$ is
the weak solution of \eqref{PrStateEq}, and
$\varphi_{k,u+\be_Y,\be_k}$ is the adjoint state of $y_{u+\be_Y}$
with respect to the $k$-th parametric control problem in
\eqref{PrKthPrlm} that is the weak solution of the following
equation
\begin{equation}\label{}
\begin{cases}
\begin{aligned}
    \mA^*\varphi+\frac{\partial f}{\partial y}(x,y_{u+\be_Y})\varphi
            &=\frac{\partial L_k}{\partial y}(x,y_{u+\be_Y})+\be_{k,J}\ &&\mbox{in}\ \Omega\\
    \varphi &=0                                                         &&\mbox{on}\ \Gamma.
\end{aligned}
\end{cases}
\end{equation}
Then, we have $\nabla_{u_k}\mJ_k(u_k,u_{-k},\be_Y,\be_k)$ can be
represented as follows
\begin{equation}\label{}
    \nabla_{u_k}\mJ_k(u_k,u_{-k},\be_Y,\be_k)
    =\zeta_ku_k+\varphi_{k,u+\be_Y,\be_k},\
    \forall k=1,\ldots,m.
\end{equation}
We now define the map
\begin{equation}\label{}
\begin{aligned}
    T:\mathbf{X}&\rightrightarrows\mathbf{X}\\
      \mathbf{u}&\mapsto\zeta\odot\mathbf{u}+N(\mathbf{u};\mU_{ad}(\mathbf{\be})).
\end{aligned}
\end{equation}
Denote $D(T)=\{\mathbf{u}\in\mathbf{X}\st
T(\mathbf{u})\neq\emptyset\}$ and choose the set
$\mG\equiv\mathbf{X}$. Then, $D(T)\equiv\mU_{ad}(\mathbf{\be})$ and
the operator $C$ defined by
\begin{equation}\label{}
\begin{aligned}
    C:\cl\,\mG\equiv\mathbf{X}&\to\mathbf{X}\\
      \mathbf{u}&\mapsto\varphi_{\mathbf{u},\mathbf{\be}}-\mathbf{\ou}^*
\end{aligned}
\end{equation}
is a compact operator. Arguing similarly to the proof of
Theorem~\ref{ThmExNshEq} we obtain the assertion of the theorem.
$\hfill\Box$
\end{proof}

\begin{Remark}\rm
Similar to Theorem~\ref{ThmCvNshEq}, if the parametric optimal
control problems considered in Theorem~\ref{ThmPaNshEq} are all
convex, then variational Nash equilibriums to the parametric system
\eqref{PaKthPVI} are also local Nash equilibriums to the equilibrium
problem associated to the optimal control problems \eqref{PrKthPrlm}
with respect to the parameters $(\mathbf{\ou}^*,\mathbf{\be})$.
\end{Remark}

\section{Full stability of variational Nash equilibriums}
\setcounter{equation}{0}

In this section, we will investigate the Lipschitzian and
H\"olderian full stability for variational Nash equilibriums to the
system \eqref{KthPVI} via the parametric system \eqref{PerKthPVI}.
We will apply the results on full stability for parametric
variational systems provided in \cite{MorNgh16SIOPT} to establish
explicit characterizations of Lipschitzian and H\"olderian full
stability for parametric variational Nash equilibriums to the system
\eqref{PerKthPVI}. Related to the solution stability for the optimal
control problem \eqref{KthCPrlm}--\eqref{SOdrEOper} via the
perturbed problem \eqref{PrKthPrlm}--\eqref{NshAdSetK} with $m=1$
and $\zeta_1=0$, we refer the reader to \cite{QuiWch18OPTIM}. In the
setting of \cite{QuiWch18OPTIM}, stability for bang-bang optimal
controls was investigated by means of tools of the optimal control
theory.

To apply the stability results of \cite{MorNgh16SIOPT}, we need to
reformulate the parametric system \eqref{PerKthPVI} to a
corresponding parametric variational inequality \eqref{NshPVInq}
below. Let us define the operator
$F:\mathbf{L}^2(\Omega)\times\mathbf{E}\rightarrow\mathbf{L}^2(\Omega)$
by setting
\begin{equation}\label{FctnFue}
    F(\mathbf{u},\mathbf{e})
    =\big(\nabla_{u_1}\mJ_1(\mathbf{u},e_Y,e_1),\ldots,\nabla_{u_m}\mJ_m(\mathbf{u},e_Y,e_m)\big),\
    \forall(\mathbf{u},\mathbf{e})\in\mathbf{L}^2(\Omega)\times\mathbf{E}.
\end{equation}
Then, for each $\mathbf{e}\in\mathbf{E}$, we have
$F'_{\mathbf{u}}(\mathbf{u},\mathbf{e})(\cdot):\mathbf{L}^2(\Omega)\to\mathbf{L}^2(\Omega)$
(or equivalently, we can rewrite as follows $\langle
F'_{\mathbf{u}}(\mathbf{u},\mathbf{e})(\cdot),\cdot\rangle:
\mathbf{L}^2(\Omega)\times\mathbf{L}^2(\Omega)\to\R$) with
$$F'_{\mathbf{u}}(\mathbf{u},\mathbf{e})\mathbf{v}\in\mathbf{L}^2(\Omega)\quad\mbox{and}\quad
  \langle F'_\mathbf{u}(\mathbf{u},\mathbf{e})\mathbf{v},\mathbf{h}\rangle
  =\sum_{k=1}^m\sum_{j=1}^m\nabla^2_{u_ku_j}\mJ_k(u_1,\ldots,u_m,e_Y,e_k)v_kh_j.$$
Using the notations given above, we consider the parametric
variational inequality associated to the system \eqref{PerKthPVI} as
follows
\begin{equation}\label{NshPVInq}
    \mathbf{u}^*\in F(\mathbf{u},\mathbf{e})+N(\mathbf{u};\mU_{ad}(\mathbf{e})).
\end{equation}
The solution map $\mS:\mathbf{L}^2(\Omega)\times
\mathbf{E}\rightrightarrows\mathbf{L}^2(\Omega)$ of the PVI
\eqref{NshPVInq} is defined by
\begin{equation}\label{NshSolMap}
\begin{aligned}
    \mS(\mathbf{u}^*,\mathbf{e})
    &=\{\mathbf{u}\in\mU_{ad}(\mathbf{e})\st
      \mathbf{u}^*\in F(\mathbf{u},\mathbf{e})+N(\mathbf{u};\mU_{ad}(\mathbf{e}))\}\\
    &=\{\mathbf{u}\in\mU_{ad}(\mathbf{e})\st
      \mathbf{u}^*\in F(\mathbf{u},\mathbf{e})+\mN(\mathbf{u},\mathbf{e})\},
\end{aligned}
\end{equation}
where the normal cone mapping
$\mN:\mathbf{L}^2(\Omega)\times\mathbf{E}\rightrightarrows\mathbf{L}^2(\Omega)$
is defined by setting
\begin{equation}\label{NshNrCnMp}
    \mN(\mathbf{u},\mathbf{e})=N(\mathbf{u};\mU_{ad}(\mathbf{e})),\
    \forall(\mathbf{u},\mathbf{e})\in\mathbf{L}^2(\Omega)\times\mathbf{E}.
\end{equation}
Following \cite{MorNgh16SIOPT}, we now recall the concepts of
Lipschitzian and H\"olderian full stability for the PVI
\eqref{NshPVInq}.

\begin{Definition}\rm
Let
$\mathbf{\ou}=(\ou_1,\ldots,\ou_m)\in\mS(\mathbf{\ou}^*,\mathbf{\be})$
from \eqref{NshSolMap} with
$(\mathbf{\ou}^*,\mathbf{\be})\in\mathbf{L}^2(\Omega)\times\mathbf{E}$.

$\bullet$ We say that $\mathbf{\ou}$ is a \emph{Lipschitzian fully
stable solution to} the PVI \eqref{NshPVInq} corresponding to the
pair $(\mathbf{\ou}^*,\mathbf{\be})$ if the solution map
\eqref{NshSolMap} admits a single-valued localization $\vartheta$
relative to some neighborhood
$\mV_{\mathbf{\ou}^*}\times\mV_{\mathbf{\be}}\times\mV_{\mathbf{\ou}}$
such that for any
$(\mathbf{u}^*_1,\mathbf{e}_1),(\mathbf{u}^*_2,\mathbf{e}_2)\in\mV_{\mathbf{\ou}^*}\times\mV_{\mathbf{\be}}$
we have
\begin{equation}\label{NshFlLpCd}
   \big\|(\mathbf{u}^*_1-\mathbf{u}^*_2)-2\kappa
   \big(\vartheta(\mathbf{u}^*_1,\mathbf{e}_1)-\vartheta(\mathbf{u}^*_2,\mathbf{e}_2)\big)\big\|_{\mathbf{L}^2(\Omega)}
   \leq\|\mathbf{u}^*_1-\mathbf{u}^*_2\|_{\mathbf{L}^2(\Omega)}+\ell\|\mathbf{e}_1-\mathbf{e}_2\|_{\mathbf{E}}
\end{equation}
with some constants $\kappa>0$ and $\ell>0$.

$\bullet$ We say that $\mathbf{\ou}$ is a \emph{H\"olderian fully
stable solution to} the PVI \eqref{NshPVInq} corresponding to the
pair $(\mathbf{\ou}^*,\mathbf{\be})$ if the solution map
\eqref{NshSolMap} admits a single-valued localization $\vartheta$
relative to some neighborhood
$\mV_{\mathbf{\ou}^*}\times\mV_{\mathbf{\be}}\times\mV_{\mathbf{\ou}}$
such that for any
$(\mathbf{u}^*_1,\mathbf{e}_1),(\mathbf{u}^*_2,\mathbf{e}_2)\in\mV_{\mathbf{\ou}^*}\times\mV_{\mathbf{\be}}$
we have
\begin{equation}\label{NshFlHdrCd}
   \big\|(\mathbf{u}^*_1-\mathbf{u}^*_2)-2\kappa
   \big(\vartheta(\mathbf{u}^*_1,\mathbf{e}_1)-\vartheta(\mathbf{u}^*_2,\mathbf{e}_2)\big)\big\|_{\mathbf{L}^2(\Omega)}
   \leq\|\mathbf{u}^*_1-\mathbf{u}^*_2\|_{\mathbf{L}^2(\Omega)}+\ell\|\mathbf{e}_1-\mathbf{e}_2\|^{1/2}_{\mathbf{E}}
\end{equation}
with some constants $\kappa>0$ and $\ell>0$.
\end{Definition}

Note that the concepts of full stability for parametric variational
systems provided in \cite{MorNgh16SIOPT} are really extensions of
the ones for local minimizers of parametric optimization problems
given in \cite{LePoRo00SIOPT} and \cite{MorNgh14SIOPT}, and they are
very effective for studying solution stability for parametric
variational inequalities.

By applying the results of the previous section we deduce that the
PVI \eqref{NshPVInq} has solutions under our standard assumptions.
Beside the assumptions {\rm\textbf{(A1)}--\textbf{(A3)}} mentioned
above, the following condition with respect to the reference
parameter $\mathbf{\be}\in\mathbf{E}$ is also a crucial role in our
investigation:

\smallskip
\textbf{(A4)} There exists $\varrho>0$ such that
$\mU_{ad}(\mathbf{e})\neq\emptyset$ for all
$\mathbf{e}\in\oB_\varrho(\mathbf{\be})\subset\mathbf{E}$.

\smallskip
We see that when $\mU_{ad}$ and $\mU_{ad}(\mathbf{e})$ are defined
via $\mU^k_{ad}$ and $\mU^k_{ad}(e_k)$ for $k=1,\ldots,m$ given in
\eqref{ExACtrSet} and \eqref{NshAdSetK} respectively, we can use the
following condition as a sufficient condition to ensure that the
assumption \textbf{(A4)} holds:
\begin{equation}\label{SpEmtyCnd}
\begin{cases}
    \exists\sigma>0,\forall k=1,\ldots,m,\\
    \alpha_k(x)+\be_{k,\alpha}(x)+\sigma
    \leq\beta_k(x)+\be_{k,\beta}(x)~\mbox{for a.a.}~x\in\Omega.
\end{cases}
\end{equation}
Related to the condition \eqref{SpEmtyCnd} we refer the reader to
\cite{QuiDWa19SICON} for more details in applications to the full
stability for local minimizers of parametric optimal control
problems.

\subsection{General case for $\mU_{ad}(\mathbf{e})$}

In this subsection, we will establish conditions for the
Lipschitzian and H\"olderian full stability for variational Nash
equilibriums (solutions) to the PVI \eqref{NshPVInq} via the data
with respect to the general convex, closed and bounded admissible
control set $\mU_{ad}(\mathbf{e})\subset\mathbf{L}^2(\Omega)$ for
$\mathbf{e}\in\mathbf{E}$.

\begin{Theorem}\label{ThmMNLpCh}
Let $\mathbf{\ou}\in\mS(\mathbf{\ou}^*,\mathbf{\be})$ from
\eqref{NshSolMap} with
$(\mathbf{\ou}^*,\mathbf{\be})\in\mathbf{L}^2(\Omega)\times\mathbf{E}$
and $\mathbf{\hu}^*:=\mathbf{\ou}^*-F(\mathbf{\ou},\mathbf{\be})$.
Assume that the assumptions {\rm\textbf{(A1)}--\textbf{(A4)}} hold.
Then, $\mathbf{\ou}$ is Lipschitzian fully stable solution to the
PVI \eqref{NshPVInq} if and only if the both conditions below hold:
\begin{itemize}
\item[{\rm(i)}] There exist $\eta>0$, $\kappa>0$ such that for $(\mathbf{u},\mathbf{e},\mathbf{u}^*)\in\gph\mN\cap
B_\eta(\mathbf{\ou},\mathbf{\be},\mathbf{\hu}^*)$ and
$\mathbf{u}^{**}\in\mathbf{L}^2(\Omega)$ we have
\begin{equation}\label{}
    \langle F'_{\mathbf{u}}(\mathbf{\ou},\mathbf{\be})\mathbf{u}^{**},\mathbf{u}^{**}\rangle
    +\langle\mathbf{v},\mathbf{u}^{**}\rangle\geq\kappa\|\mathbf{u}^{**}\|_{\mathbf{L}^2(\Omega)},\
    \forall\mathbf{v}\in\widehat{D}^*\mN_\mathbf{e}(\mathbf{u},\mathbf{u}^*)(\mathbf{u}^{**}),
\end{equation}
where $\mN_\mathbf{e}(\cdot)=\mN(\cdot,\mathbf{e})$ from
\eqref{NshNrCnMp}.
\item[{\rm(ii)}] The graphical mapping $\mathbf{e}\mapsto\gph\mN(\cdot,\mathbf{e})$ is
locally Lipschitz-like around
$(\mathbf{\be},\mathbf{\ou},\mathbf{\hu}^*)$.
\end{itemize}
\end{Theorem}
\begin{proof}
Using the assumptions {\rm\textbf{(A1)}--\textbf{(A3)}} and applying
Theorem~\ref{ThmFSdDerJ} we deduce that $F(\cdot,\cdot)$ given in
\eqref{FctnFue} is differentiable with respect to $\mathbf{u}$
around $(\mathbf{\ou},\mathbf{\be})$ uniformly in $\mathbf{e}$ and
the partial derivative $F'_\mathbf{u}(\cdot,\cdot)$ is continuous at
$(\mathbf{\ou},\mathbf{\be})$. Moreover, $F(\cdot,\cdot)$ is also
Lipschitz continuous with respect to $\mathbf{e}$ uniformly in
$\mathbf{u}$ around $(\mathbf{\ou},\mathbf{\be})$. In addition, by
arguing similarly to the proof of \cite[Theorem~4.1]{QuiDWa19SICON}
we can verify that the parametric indicator
$\delta(\cdot\,;\mU_{ad}(\cdot))$ is parametrically continuously
prox-regular at $(\mathbf{\ou},\mathbf{\be})$ for $\mathbf{\hu}^*$
and the basic constraint qualification (BCQ) holds at
$(\mathbf{\ou},\mathbf{\be})$ under the assumption \textbf{(A4)};
see the parametric continuous prox-regularity and the BCQ in
\cite{MoOuSa14SIOPT}. Summarizing the above, we infer that all the
assumptions stated in \cite[Theorem~4.7]{MorNgh16SIOPT} are
satisfied. Therefore, applying \cite[Theorem~4.7]{MorNgh16SIOPT} we
obtain the assertion of the theorem. $\hfill\Box$
\end{proof}

\begin{Theorem}\label{ThmMNHdrCh}
Let $\mathbf{\ou}\in\mS(\mathbf{\ou}^*,\mathbf{\be})$ from
\eqref{NshSolMap} with
$(\mathbf{\ou}^*,\mathbf{\be})\in\mathbf{L}^2(\Omega)\times\mathbf{E}$
and $\mathbf{\hu}^*:=\mathbf{\ou}^*-F(\mathbf{\ou},\mathbf{\be})$.
Assume that the assumptions {\rm\textbf{(A1)}--\textbf{(A4)}} hold.
Let us consider the following two statements:
\begin{itemize}
\item[{\rm(i)}] $\mathbf{\ou}$ is H\"olderian fully stable solution to the PVI
\eqref{NshPVInq} corresponding to the parameter pair
$(\mathbf{\ou}^*,\mathbf{\be})$ with the moduli $\kappa>0$ and
$\ell>0$ taken from \eqref{NshFlHdrCd}.
\item[{\rm(ii)}] There exist some $\eta>0$, $\kappa_0>0$ such that for
$(\mathbf{u},\mathbf{e},\mathbf{u}^*)\in\gph\mN\cap
B_\eta(\mathbf{\ou},\mathbf{\be},\mathbf{\hu}^*)$ we have
\begin{equation}\label{}
    \langle F'_{\mathbf{u}}(\mathbf{\ou},\mathbf{\be})\mathbf{u}^{**},\mathbf{u}^{**}\rangle
    +\langle\mathbf{v},\mathbf{u}^{**}\rangle\geq\kappa_0\|\mathbf{u}^{**}\|_{\mathbf{L}^2(\Omega)},\
    \forall\mathbf{v}\in\widehat{D}^*\mN_\mathbf{e}(\mathbf{u},\mathbf{u}^*)(\mathbf{u}^{**}).
\end{equation}
\end{itemize}
Then, {\rm(i)} implies {\rm(ii)} with constant $\kappa_0$ that can
be chosen smaller than but arbitrarily closed to $\kappa$.
Conversely, the validity of {\rm(ii)} ensures that {\rm(i)} holds,
where $\kappa$ can be chosen smaller than but arbitrarily closed to
$\kappa_0$.
\end{Theorem}
\begin{proof}
Arguing similarly to the proof of Theorem~\ref{ThmMNLpCh}, we can
check that all the assumptions of \cite[Theorem~4.3]{MorNgh16SIOPT}
hold. Therefore, applying \cite[Theorem~4.3]{MorNgh16SIOPT} we
obtain the assertions of the theorem. $\hfill\Box$
\end{proof}

\bigskip
Following \cite{Har77JMSJ}, we say that a closed and convex subset
$\Theta$ of a Banach space $X$ is \emph{polyhedric} at
$\ou\in\Theta$ for $\hu^*\in N(\ou;\Theta)$ if we have the
representation
\begin{equation}\label{CndPlyhrc}
    T(\ou;\Theta)\cap\{\hu^*\}^\bot=\cl\big(\cone(\Theta-\ou)\cap\{\hu^*\}^\bot\big),
\end{equation}
where
\begin{equation}\label{}
    \cone(\Theta-\ou)=\bigcup_{t>0}\frac{\Theta-\ou}{t}
\end{equation}
is the radial cone, and
\begin{equation}\label{}
    T(\ou;\Theta)=\cl(\cone(\Theta-\ou))
\end{equation}
is the tangent cone to $\Theta$ at $\ou$. The set $\Theta$ is said
to be \emph{polyhedric} if $\Theta$ is polyhedric at every
$u\in\Theta$ for any $u^*\in N(u;\Theta)$. The polyhedricity
property of a set is first introduced in \cite{Har77JMSJ} and then
applied extensively in optimal control; see, e.g.,
\cite{Bon98AMO,BonSha00B,ItoKun08B} and the references therein.

\begin{Theorem}{\rm(See \cite[Theorem~6.2]{MorNgh14SIOPT})}\label{ThmMNFrCdr}
Let
$(\mathbf{\ou},\mathbf{\be})\in\mU_{ad}(\mathbf{\be})\times\mathbf{E}$
and let $\mathbf{\hu}^*\in N(\mathbf{\ou};\mU_{ad}(\mathbf{\be}))$.
Then, we have
\begin{equation}\label{DomCmBnSbd}
    \dom\widehat{D}^*\mN_\mathbf{\be}(\mathbf{\ou},\mathbf{\hu}^*)
    \subset-T(\mathbf{\ou};\mU_{ad}(\mathbf{\be}))\cap\{\mathbf{\hu}^*\}^\bot.
\end{equation}
If, in addition, $\mU_{ad}(\mathbf{\be})$ is polyhedric at
$\mathbf{\ou}\in\mU_{ad}(\mathbf{\be})$ for $\mathbf{\hu}^*$, then
the equality
\begin{equation}\label{EqCmBnSbdf}
    \widehat{D}^*\mN_\mathbf{\be}(\mathbf{\ou},\mathbf{\hu}^*)(\mathbf{u})
    =\big(T(\mathbf{\ou};\mU_{ad}(\mathbf{\be}))\cap\{\mathbf{\hu}^*\}^\bot\big)^*
\end{equation}
holds for all
$\mathbf{u}\in-T(\mathbf{\ou};\mU_{ad}(\mathbf{\be}))\cap\{\mathbf{\hu}^*\}^\bot$.
\end{Theorem}

For each
$(\mathbf{u},\mathbf{e},\mathbf{u}^*)\in\mathbf{L}^2(\Omega)\times\mathbf{E}\times\mathbf{L}^2(\Omega)$
with $\mathbf{u}^*\in\mN(\mathbf{u},\mathbf{e})$, we define the
critical cone
\begin{equation}\label{NshCrCone}
    \mC(\mathbf{u},\mathbf{e},\mathbf{u}^*)=T(\mathbf{u};\mU_{ad}(\mathbf{e}))\cap\{\mathbf{u}^*\}^\bot.
\end{equation}
Note that
$\mU_{ad}(\mathbf{e})=\mU^1_{ad}(e_1)\times\cdots\times\mU^m_{ad}(e_m)\subset\mathbf{L}^2(\Omega)$
with $\mU^k_{ad}(e_k)\subset L^2(\Omega)$ being convex and $u^*_k\in
N(u_k;\mU^k_{ad}(e_k))$ for every $k\in\{1,\ldots,m\}$. Therefore,
from \eqref{NshCrCone} we obtain
\begin{equation}\label{NshExCone}
    \mC(\mathbf{u},\mathbf{e},\mathbf{u}^*)
    =\prod_{k=1}^m\mC(u_k,e_k,u^*_k)
    =\prod_{k=1}^mT(u_k;\mU^k_{ad}(e_k))\cap\{u^*_k\}^\bot,
\end{equation}
where $\mC(u_k,e_k,u^*_k):=T(u_k;\mU^k_{ad}(e_k))\cap\{u^*_k\}^\bot$
for every $k\in\{1,\ldots,m\}$. Let us now define the sequential
outer limits of the critical cones
$\mC(\mathbf{u},\mathbf{e},\mathbf{u}^*)$ respectively in the weak*
topology of $\mathbf{L}^2(\Omega)$ by
\begin{equation}\label{LmWStopoAd}
    \mC_{w^*}(\mathbf{\ou},\mathbf{\be},\mathbf{\hu}^*)
    =\Limsup_{(\mathbf{u},\mathbf{e},\mathbf{u}^*)\stackrel{\gph\mN}
    \longrightarrow(\mathbf{\ou},\mathbf{\be},\mathbf{\hu}^*)}\mC(\mathbf{u},\mathbf{e},\mathbf{u}^*),
\end{equation}
and in the strong topology of $\mathbf{L}^2(\Omega)$ as follows
\begin{equation}\label{NsLmSTtpA}
    \mC_s(\mathbf{\ou},\mathbf{\be},\mathbf{\hu}^*)
    =\Big\{\mathbf{v}\in\mathbf{L}^2(\Omega)\Bst\exists(\mathbf{u}_n,\mathbf{e}_n,\mathbf{u}^*_n)
    \stackrel{\gph\mN}\longrightarrow(\mathbf{\ou},\mathbf{\be},\mathbf{\hu}^*),
    \mathbf{v}_n\in\mC(\mathbf{u}_n,\mathbf{e}_n,\mathbf{u}^*_n),\mathbf{v}_n\to\mathbf{v}\Big\}.
\end{equation}
Then, using \eqref{NshExCone} we deduce from \eqref{LmWStopoAd} and
\eqref{NsLmSTtpA} that
\begin{equation}\label{NshWeStCn}
    \mC_{w^*}(\mathbf{\ou},\mathbf{\be},\mathbf{\hu}^*)
    =\prod_{k=1}^m\mC_{w^*}(\ou_k,\be_k,\hu^*_k)
    \quad\mbox{and}\quad
    \mC_s(\mathbf{\ou},\mathbf{\be},\mathbf{\hu}^*)
    =\prod_{k=1}^m\mC_s(\ou_k,\be_k,\hu^*_k),
\end{equation}
where
\begin{equation}\label{}
    \mC_{w^*}(\ou_k,\be_k,\hu^*_k)
    =\Limsup_{(u,e,u^*)\stackrel{\gph\mN_k}
    \longrightarrow(\ou_k,\be_k,\ou^*_k)}\mC(u,e,u^*)
\end{equation}
and
\begin{equation}\label{}
    \mC_s(\ou_k,\be_k,\hu^*_k)
    =\Big\{v\in L^2(\Omega)\Bst\exists(u_{kn},e_{kn},u^*_{kn})
    \stackrel{\gph\mN_k}\longrightarrow(\ou_k,\be_k,\hu^*_k),
    v_{kn}\in\mC(u_{kn},e_{kn},u^*_{kn}),v_{kn}\to v\Big\}
\end{equation}
with $\mN_k(u,e)=N(u;\mU^k_{ad}(e))$ for all $(u,e)\in
L^2(\Omega)\times E_k$.

\begin{Definition}\rm
A quadratic form $Q:H\to\R$ on a Hilbert space $H$ is said to be a
\emph{Legendre form} if $Q$ is sequentially weakly lower
semicontinuous and that if $h_n$ converges weakly to $h$ in $H$ and
$Q(h_n)\to Q(h)$ then $h_n$ converges strongly to $h$ in $H$.
\end{Definition}

\begin{Theorem}\label{ThmNcSfCd}
Let
$\mathbf{\ou}=(\ou_1,\ldots,\ou_m)\in\mS(\mathbf{\ou}^*,\mathbf{\be})$
from \eqref{NshSolMap} and
$\mathbf{\hu}^*:=\mathbf{\ou}^*-F(\mathbf{\ou},\mathbf{\be})$.
Assume that the assumptions {\rm\textbf{(A1)}--\textbf{(A4)}} hold.
Let us define the quadratic form $\mQ:\mathbf{L}^2(\Omega)\to\R$ by
\begin{equation}\label{NsDfQdrFm}
    \mQ(\mathbf{h}):=\langle F'_\mathbf{u}(\mathbf{\ou},\mathbf{\be})\mathbf{h},\mathbf{h}\rangle,\
    \forall\mathbf{h}\in\mathbf{L}^2(\Omega).
\end{equation}
Then, the following assertions are valid:
\begin{itemize}
\item[{\rm(i)}] If the quadratic form $\mQ(\cdot)$
is a Legendre form on $\mathbf{L}^2(\Omega)$ and
\begin{equation}\label{NsStrPsDf}
    \langle F'_\mathbf{u}(\mathbf{\ou},\mathbf{\be})\mathbf{v},\mathbf{v}\rangle>0,\
    \forall\mathbf{v}\in\mC_{w^*}(\mathbf{\ou},\mathbf{\be},\mathbf{\hu}^*)~\mbox{with}~\mathbf{v}\neq0,
\end{equation}
then $\mathbf{\ou}$ is a fully stable solution to the PVI
\eqref{NshPVInq}.
\item[{\rm(ii)}] If $\mathbf{\ou}$ is a fully stable solution to the PVI
\eqref{NshPVInq} and that $\mU_{ad}(\mathbf{e})$ is polyhedric
around $\mathbf{\be}$, then we have the positive definiteness
condition
\begin{equation}\label{NsWkPDfCd}
    \langle F'_\mathbf{u}(\mathbf{\ou},\mathbf{\be})\mathbf{v},\mathbf{v}\rangle>0,\
    \forall\mathbf{v}\in\mC_s(\mathbf{\ou},\mathbf{\be},\mathbf{\hu}^*)~\mbox{with}~\mathbf{v}\neq0.
\end{equation}
\end{itemize}
\end{Theorem}
\begin{proof}
To prove the assertion (i) we suppose to the contrary that
$\mathbf{\ou}$ is not a fully stable solution to the PVI
\eqref{NshPVInq}. By Theorem~\ref{ThmMNLpCh}, one can find some
sequences
$(\mathbf{u}_n,\mathbf{e}_n,\mathbf{u}^*_n)\to(\mathbf{\ou},\mathbf{\be},\mathbf{\hu}^*)$
with $(\mathbf{u}_n,\mathbf{e}_n,\mathbf{u}^*_n)\in\gph\mN$ and
$(\mathbf{u}^{**}_n,\mathbf{v}_n)\in\mathbf{L}^2(\Omega)\times\mathbf{L}^2(\Omega)$
with
$\mathbf{v}_n\in\widehat{D}^*\mN_{\mathbf{e}_n}(\mathbf{u}_n,\mathbf{u}^*_n)(\mathbf{u}^{**}_n)$
such that
\begin{equation}\label{NsCntrCd}
    \langle F'_\mathbf{u}(\mathbf{\ou},\mathbf{\be})\mathbf{u}^{**}_n,\mathbf{u}^{**}_n\rangle
    +\langle\mathbf{v}_n,\mathbf{u}^{**}_n\rangle<\frac{1}{n}\|\mathbf{u}^{**}_n\|^2_{\mathbf{L}^2(\Omega)},\
    \forall n\in\N.
\end{equation}
Since
$\mN_{\mathbf{e}_n}:\mathbf{L}^2(\Omega)\times\mathbf{E}\rightrightarrows\mathbf{L}^2(\Omega)$
is maximal monotone, by \cite[Lemma~3.3]{ChiTra12TJM} we get
$\langle\mathbf{v}_n,\mathbf{u}^{**}_n\rangle\geq0$. According to
Theorem~\ref{ThmMNFrCdr}, we have
$\mathbf{u}^{**}_n\in-\mC(\mathbf{u}_n,\mathbf{e}_n,\mathbf{u}^*_n)$,
which yields $\mathbf{u}^{**}_n\neq0$. Combining this and
\eqref{NsCntrCd} with
$\langle\mathbf{v}_n,\mathbf{u}^{**}_n\rangle\geq0$, we have
\begin{equation}\label{NsCtrCdMdf}
    \langle F'_\mathbf{u}(\mathbf{\ou},\mathbf{\be})\mathbf{u}^{**}_n,\mathbf{u}^{**}_n\rangle
    <\frac{1}{n}\|\mathbf{u}^{**}_n\|^2_{\mathbf{L}^2(\Omega)}
    \quad\mbox{with}~~\mathbf{u}^{**}_n\in-\mC(\mathbf{u}_n,\mathbf{e}_n,\mathbf{u}^*_n).
\end{equation}
By setting
$\mathbf{\wu}^{**}_n:=\mathbf{u}^{**}_n\|\mathbf{u}^{**}_n\|^{-1}_{\mathbf{L}^2(\Omega)}$,
we deduce from \eqref{NsCtrCdMdf} that
\begin{equation}\label{NsCtrNwMdf}
    \mQ(\mathbf{\wu}^{**}_n)
    =\langle F'_\mathbf{u}(\mathbf{\ou},\mathbf{\be})\mathbf{\wu}^{**}_n,\mathbf{\wu}^{**}_n\rangle
    <\frac{1}{n}\quad\mbox{with}~~\mathbf{\wu}^{**}_n\in-\mC(\mathbf{u}_n,\mathbf{e}_n,\mathbf{u}^*_n).
\end{equation}
We may assume that $\mathbf{\wu}^{**}_n$ weakly converges to some
$\mathbf{\wu}^{**}$. It follows from \eqref{LmWStopoAd} that
\begin{equation}\label{NsCdBgCne}
    \mathbf{\wu}^{**}\in-\mC_{w^*}(\mathbf{\ou},\mathbf{\be},\mathbf{\hu}^*)
    \quad\mbox{or equivalently as}\quad-\mathbf{\wu}^{**}\in\mC_{w^*}(\mathbf{\ou},\mathbf{\be},\mathbf{\hu}^*).
\end{equation}
From the weak lower semicontinuity of $\mQ$ and from
\eqref{NsDfQdrFm}, \eqref{NsStrPsDf}, \eqref{NsCtrNwMdf} and
\eqref{NsCdBgCne} it follows that
\begin{equation}\label{}
    0\leq\mQ(-\mathbf{\wu}^{**})=\mQ(\mathbf{\wu}^{**})\leq\liminf_{n\to\infty}\mQ(\mathbf{\wu}^{**}_n)\leq0,
\end{equation}
which yields $\mQ(-\mathbf{\wu}^{**})=0$ with
$-\mathbf{\wu}^{**}\in\mC_{w^*}(\mathbf{\ou},\be,\mathbf{\hu}^*)$.
Since $\mQ(\cdot)$ is a Legendre form on $\mathbf{L}^2(\Omega)$, we
obtain $\mathbf{\wu}^{**}_n\to\mathbf{\wu}^{**}$. This implies that
$\|\mathbf{\wu}^{**}\|_{\mathbf{L}^2(\Omega)}=1$, and thus
$-\mathbf{u}^{**}\neq0$. We have arrived at a contradiction.

We now prove the assertion (ii). By Theorem~\ref{ThmMNLpCh}, there
exist $\kappa>0$, $\eta>0$ such that for any
$(\mathbf{u},\mathbf{e},\mathbf{u}^*)\in\gph\mN\cap
B_\eta(\mathbf{\ou},\mathbf{\be},\mathbf{\hu}^*)$ and
$\mathbf{u}^{**}\in\mathbf{L}^2(\Omega)$ we have
\begin{equation}\label{}
    \langle F'_\mathbf{u}(\mathbf{\ou},\mathbf{\be})\mathbf{u}^{**},\mathbf{u}^{**}\rangle
    +\langle\mathbf{v},\mathbf{u}^{**}\rangle\geq\kappa\|\mathbf{u}^{**}\|_{\mathbf{L}^2(\Omega)},\
    \forall\mathbf{v}\in\widehat{D}^*\mN_\mathbf{e}(\mathbf{u},\mathbf{u}^*)(\mathbf{u}^{**}).
\end{equation}
Since $\mU_{ad}(\mathbf{e})$ is polyhedric around $\mathbf{\be}$, by
Theorem~\ref{ThmMNFrCdr} we deduce that
$$\mathbf{u}^{**}\in-\mC(\mathbf{u},\mathbf{e},\mathbf{u}^*)
  \quad\mbox{and}\quad\widehat{D}^*\mN_\mathbf{e}(\mathbf{u},\mathbf{u}^*)(\mathbf{u}^{**})
  =\mC(\mathbf{u},\mathbf{e},\mathbf{u}^*)^*,$$
which yields
$0\in\widehat{D}^*\mN_\mathbf{e}(\mathbf{u},\mathbf{u}^*)(\mathbf{u}^{**})$.
Consequently, for all
$(\mathbf{u},\mathbf{e},\mathbf{u}^*)\in\gph\mN\cap
B_\eta(\mathbf{\ou},\mathbf{\be},\mathbf{\hu}^*)$, we have
\begin{equation}\label{NsLwSPsCd}
    \langle F'_\mathbf{u}(\mathbf{\ou},\mathbf{\be})\mathbf{u}^{**},\mathbf{u}^{**}\rangle
    \geq\kappa\|\mathbf{u}^{**}\|_{\mathbf{L}^2(\Omega)},\
    \forall\mathbf{u}^{**}\in-\mC(\mathbf{u},\mathbf{e},\mathbf{u}^*).
\end{equation}
Using the strong convergence in \eqref{NsLmSTtpA} and passing
\eqref{NsLwSPsCd} to the limit when $\eta\downarrow0$ we obtain the
positive definiteness condition \eqref{NsWkPDfCd}. $\hfill\Box$
\end{proof}

\medskip
The following lemma shows that the quadratic $\mQ(\cdot)$ defined by
\eqref{NsDfQdrFm} is a Legendre form on the space
$\mathbf{L}^2(\Omega)$. This is one of important results that will
help us to establish explicit characterizations of full stability
for parametric variational Nash equilibriums.

\begin{Lemma}\label{LmQLgrFm}
Assume that the assumptions {\rm\textbf{(A1)}--\textbf{(A4)}} hold.
Let
$(\mathbf{\ou},\mathbf{\be})\in\mathbf{L}^2(\Omega)\times\mathbf{E}$
and define the quadratic form $\mQ:\mathbf{L}^2(\Omega)\to\R$ by
\begin{equation}\label{DfLgdrmQ}
    \mQ(\mathbf{h})=\langle F'_\mathbf{u}(\mathbf{\ou},\mathbf{\be})\mathbf{h},\mathbf{h}\rangle,\
    \forall\mathbf{h}\in\mathbf{L}^2(\Omega).
\end{equation}
Then, $\mQ(\cdot)$ is a Legendre form on $\mathbf{L}^2(\Omega)$.
\end{Lemma}
\begin{proof}
For $\mathbf{h}=(h_1,\ldots,h_m)\in\mathbf{L}^2(\Omega)$, we have
\begin{equation}\label{mQExpFrm}
    \mQ(\mathbf{h})=\langle F'_\mathbf{u}(\mathbf{\ou},\mathbf{\be})\mathbf{h},\mathbf{h}\rangle
    =\sum_{k=1}^m\sum_{j=1}^m\nabla^2_{u_ku_j}\mJ_k(\ou_1,\ldots,\ou_m,\be_Y,\be_k)h_kh_j.
\end{equation}
For $k,j\in\{1,\ldots,m\}$, we define $Q_{kj}:L^2(\Omega)\times
L^2(\Omega)\to\R$ by
\begin{equation}\label{DefQdrFm}
    Q_{kj}(v,h)=\nabla^2_{u_ku_j}\mJ_k(\ou_1,\ldots,\ou_m,\be_Y,\be_k)(v,h),\
    \forall(v,h)\in L^2(\Omega)\times L^2(\Omega).
\end{equation}
Then, from \eqref{mQExpFrm} we have
\begin{equation}\label{ExpmQQkj}
    \mQ(\mathbf{h})=\sum_{k=1}^m\sum_{j=1}^mQ_{kj}(h_k,h_j).
\end{equation}
We have
$$\begin{aligned}
    Q_{kj}(h_k,h_j)
    &=\nabla^2_{u_ku_j}\mJ_k(\ou_1,\ldots,\ou_m,\be_Y,\be_k)(h_k,h_j)\\
    &=\int_\Omega\bigg(\dfrac{\partial^2L_k}{\partial y^2}(x,y_{\ou+\be_Y})
     -\varphi_{k,\ou+\be_Y}\dfrac{\partial^2f}{\partial y^2}(x,y_{\ou+\be_Y})\bigg)
     z_{{\ou+\be_Y},h_k}z_{{\ou+\be_Y},h_j}dx\\
    &\quad +\int_\Omega\be_{k,J}G''(\ou+\be_Y)(h_k,h_j)dx+\int_\Omega\chi_{\{k\}}(j)\zeta_kh_kh_jdx\\
    &=Q^1_{kj}(h_k,h_j)+Q^2_{kj}(h_k,h_j),
\end{aligned}$$
where
$$\begin{aligned}
    Q^1_{kj}(h_k,h_j)
    &=\int_\Omega\bigg(\dfrac{\partial^2L_k}{\partial y^2}(x,y_{\ou+\be_Y})
     -\varphi_{k,\ou+\be_Y}\dfrac{\partial^2f}{\partial y^2}(x,y_{\ou+\be_Y})\bigg)
     z_{{\ou+\be_Y},h_k}z_{{\ou+\be_Y},h_j}dx\\
    &\quad +\int_\Omega\be_{k,J}G''(\ou+\be_Y)(h_k,h_j)dx
\end{aligned}$$ and
$$Q^2_{kj}(h_k,h_j)=\int_\Omega\chi_{\{k\}}(j)\zeta_kh_kh_jdx.$$
From \eqref{ExpmQQkj} we have
\begin{equation}\label{SumQQ1Q2}
    \mQ(\mathbf{h})=\sum_{k=1}^m\sum_{j=1}^mQ^1_{kj}(h_k,h_j)+\sum_{k=1}^mQ^2_{kk}(h_k,h_k)
    =\mQ_1(\mathbf{h})+\mQ_2(\mathbf{h}),
\end{equation}
where
$$\mQ_1(\mathbf{h})=\sum_{k=1}^m\sum_{j=1}^mQ^1_{kj}(h_k,h_j)\quad\mbox{and}
\quad\mQ_2(\mathbf{h})=\sum_{k=1}^mQ^2_{kk}(h_k,h_k).$$

Since the operator $G'(\ou+\be_Y):h\mapsto z_{\ou+\be_Y,h}$ from
$L^2(\Omega)$ into $L^2(\Omega)$ is compact, $Q^1_{kj}(h_k,h_j)$ is
weakly continuous on $L^2(\Omega)\times L^2(\Omega)$. It follows
that $\mQ_1(\mathbf{h})$ is weakly continuous on
$\mathbf{L}^2(\Omega)$. In addition, the quadratic form
\begin{equation}\label{}
    \mQ_2(\mathbf{h})=\sum_{k=1}^mQ^2_{kk}(h_k,h_k)=\sum_{k=1}^m\int_\Omega\zeta_kh_k^2dx
\end{equation}
is a Legendre form on $\mathbf{L}^2(\Omega)$.

Suppose that $\mathbf{h}_n\rightharpoonup\mathbf{h}$ in
$\mathbf{L}^2(\Omega)$ and that
$\mQ(\mathbf{h}_n)\to\mQ(\mathbf{h})$, where
$\mathbf{h}_n=(h_{1n},\ldots,h_{mn})$ and
$\mathbf{h}=(h_1,\ldots,h_m)$. Then, we have
\begin{equation}\label{hkWkCvgce}
    h_{kn}\rightharpoonup h_k,\ \forall k\in\{1,\ldots,m\}.
\end{equation}
It follows that
\begin{equation}\label{}
    \lim_{n\to\infty}Q^1_{kj}(h_{kn},h_{jn})=Q^1_{kj}(h_k,h_j),\ \forall k,j\in\{1,\ldots,m\}.
\end{equation}
Consequently, we have
\begin{equation}\label{}
    \lim_{n\to\infty}\mQ_1(\mathbf{h}_n)
    =\lim_{n\to\infty}\sum_{k=1}^m\sum_{j=1}^mQ^1_{kj}(h_{kn},h_{jn})
    =\sum_{k=1}^m\sum_{j=1}^mQ^1_{kj}(h_k,h_j)=\mQ_1(\mathbf{h}).
\end{equation}
Combining this with \eqref{SumQQ1Q2} yields
\begin{equation}\label{}
    \lim_{n\to\infty}\mQ_2(\mathbf{h}_n)
    =\lim_{n\to\infty}\Big(\mQ(\mathbf{h}_n)-\mQ_1(\mathbf{h}_n)\Big)
    =\mQ(\mathbf{h})-\mQ_1(\mathbf{h})=\mQ_2(\mathbf{h}).
\end{equation}
Since $\mQ_2(\cdot)$ is a Legendre form on $\mathbf{L}^2(\Omega)$,
we deduce that $\mathbf{h}_n\to\mathbf{h}$ as $n\to\infty$. We have
shown that the quadratic form $\mQ(\cdot)$ defined by
\eqref{DfLgdrmQ} is a Legendre form on $\mathbf{L}^2(\Omega)$.
$\hfill\Box$
\end{proof}

\begin{Lemma}[Mazur's lemma]\label{LemMzur}
Let $X$ be a Banach space and let $\{u_n\}_{n\in\N}\subset X$ be
such that $u_n$ converges weakly to some $\ou$ in $X$. Then, there
exist a function $\sigma:\N\to\N$ and a sequence $\{\varrho(n)_r\st
r=n,\ldots,\sigma(n)\}$ satisfying $\varrho(n)_r\geq0$ and
$\sum_{r=n}^{\sigma(n)}\varrho(n)_r=1$ such that for the sequence
$\{v_n\}_{n\in\N}$ defined by
$v_n=\sum_{r=n}^{\sigma(n)}\varrho(n)_ru_r$ we have $v_n$ converges
strongly to $\ou$ in $X$.
\end{Lemma}

\begin{Theorem}\label{ThmUpLwCn}
Let
$(\mathbf{\ou},\mathbf{\be},\mathbf{\ou}^*)\in\mU_{ad}(\mathbf{\be})\times\mathbf{E}\times\mathbf{L}^2(\Omega)$
such that
$\mathbf{\hu}^*:=\mathbf{\ou}^*-F(\mathbf{\ou},\mathbf{\be})\in\mN(\mathbf{\ou},\mathbf{\be})$.
Assume that the assumptions {\rm\textbf{(A1)}--\textbf{(A4)}} hold.
Then, we have
\begin{equation}\label{UpEsCone}
    \mC_{w^*}(\mathbf{\ou},\mathbf{\be},\mathbf{\hu}^*)
    \subset\prod_{k=1}^m\Big(\cl\big[T(\ou_k;\mU^k_{ad}(\be_k))-T(\ou_k;\mU^k_{ad}(\be_k))\big]
    \cap\{\hu^*_k\}^\bot\Big).
\end{equation}
If, in addition, $\mU_{ad}(\mathbf{e})$ is polyhedric around
$\mathbf{\be}$, then we have the following lower estimate
\begin{equation}\label{LwEsCone}
    \prod_{k=1}^m\Big(\mC(\ou_k,\be_k,\hu^*_k)-\mC(\ou_k,\be_k,\hu^*_k)\Big)
    \subset\mC_s(\mathbf{\ou},\mathbf{\be},\mathbf{\hu}^*).
\end{equation}
\end{Theorem}
\begin{proof}
To prove \eqref{UpEsCone} we take any
$\mathbf{v}\in\mC_{w^*}(\mathbf{\ou},\mathbf{\be},\mathbf{\hu}^*)$.
By \eqref{LmWStopoAd}, there exist some sequences
$(\mathbf{u}_n,\mathbf{e}_n,\mathbf{u}^*_n)\to(\mathbf{\ou},\mathbf{\be},\mathbf{\hu}^*)$
with $(\mathbf{u}_n,\mathbf{e}_n,\mathbf{u}^*_n)\in\gph\mN$ and
$\mathbf{v}_n\stackrel{w}\rightharpoonup\mathbf{v}$ with
$\mathbf{v}_n\in\mC(\mathbf{u}_n,\mathbf{e}_n,\mathbf{u}^*_n)$. Note
that we have
$$\mathbf{v}_n\in\mC(\mathbf{u}_n,\mathbf{e}_n,\mathbf{u}^*_n)
  =T(\mathbf{u}_n;\mU_{ad}(\mathbf{e}_n))\cap\{\mathbf{u}^*_n\}^\bot.$$ Hence,
for each $n$, there exist sequences $\mathbf{w}_{nr}\to\mathbf{u}_n$
with $\mathbf{w}_{nr}\in\mU_{ad}(\mathbf{e}_n)$ and
$t_{nr}\downarrow0$ satisfying
$t_{nr}^{-1}(\mathbf{w}_{nr}-\mathbf{u}_n)\to\mathbf{v}_n$. We
define
$(\mathbf{d}_n,\varrho_n)\in\mU_{ad}(\mathbf{e}_n)\times(0,+\infty)$
by
$$(\mathbf{d}_n,\varrho_n)\in\{(\mathbf{w}_{nr},t_{nr})\st r\in\N\}$$
such that $\mathbf{d}_n\to\mathbf{\ou}$, $\varrho_n\downarrow0$, and
$\varrho_n^{-1}(\mathbf{d}_n-\mathbf{u}_n)-\mathbf{v}_n\to0$ when
$n\to\infty$. It follows that
$$\begin{aligned}
    \mathbf{v}=\lim^w_{n\to\infty}\mathbf{v}_n=\lim^w_{n\to\infty}\frac{\mathbf{d}_n-\mathbf{u}_n}{\varrho_n}
    &=\lim^w_{n\to\infty}\frac{(\mathbf{d}_n-\mathbf{\ou})-(\mathbf{u}_n-\mathbf{\ou})}{\varrho_n}\\
    &\in\cl^w[T(\mathbf{\ou};\mU_{ad}(\mathbf{\be}))-T(\mathbf{\ou};\mU_{ad}(\mathbf{\be}))\big],
\end{aligned}$$
where ``$\lim^w_{n\to\infty}\mathbf{v}_n$'' stands for the weak
limit of the sequence $\mathbf{v}_n$ in $\mathbf{L}^2(\Omega)$.
Since $\langle\mathbf{u}^*_n,\mathbf{v}_n\rangle=0$, by passing to
the limit we get $\langle\mathbf{\hu}^*,\mathbf{v}\rangle=0$.
Therefore, we obtain
$$\begin{aligned}
    \mathbf{v}\in\cl^w[T(\mathbf{\ou};\mU_{ad}(\mathbf{\be}))-T(\mathbf{\ou};\mU_{ad}(\mathbf{\be}))\big]
    \cap\{\mathbf{\hu}^*\}^\bot.
\end{aligned}$$
Since
$T(\mathbf{\ou};\mU_{ad}(\mathbf{\be}))-T(\mathbf{\ou};\mU_{ad}(\mathbf{\be}))$
is convex, by Lemma~\ref{LemMzur} (Mazur's lemma) we deduce that
$$\cl^w[T(\mathbf{\ou};\mU_{ad}(\mathbf{\be}))-T(\mathbf{\ou};\mU_{ad}(\mathbf{\be}))\big]
  =\cl[T(\mathbf{\ou};\mU_{ad}(\mathbf{\be}))-T(\mathbf{\ou};\mU_{ad}(\mathbf{\be}))\big].$$
This implies that
$\mathbf{v}\in\cl[T(\mathbf{\ou};\mU_{ad}(\mathbf{\be}))-T(\mathbf{\ou};\mU_{ad}(\mathbf{\be}))\big]$,
which yields
\begin{equation}\label{PrvUpEsCn}
    \mC_{w^*}(\mathbf{\ou},\mathbf{\be},\mathbf{\hu}^*)
    \subset\cl\big[T(\mathbf{\ou};\mU_{ad}(\mathbf{\be}))-T(\mathbf{\ou};\mU_{ad}(\mathbf{\be}))\big]
    \cap\{\mathbf{\hu}^*\}^\bot.
\end{equation}
In addition, we can verify that
\begin{equation}\label{PrvEqEsCn}
\begin{aligned}
    &\qquad\quad\;\,\cl\big[T(\mathbf{\ou};\mU_{ad}(\mathbf{\be}))-T(\mathbf{\ou};\mU_{ad}(\mathbf{\be}))\big]
    \cap\{\mathbf{\hu}^*\}^\bot\\
    &=\prod_{k=1}^m\Big(\cl\big[T(\ou_k;\mU^k_{ad}(\be_k))-T(\ou_k;\mU^k_{ad}(\be_k))\big]
    \cap\{\hu^*_k\}^\bot\Big).
\end{aligned}
\end{equation}
Combining \eqref{PrvUpEsCn} and \eqref{PrvEqEsCn} we obtain
\eqref{UpEsCone}.

To prove \eqref{LwEsCone} we take any
$\mathbf{v}=\mathbf{v}_1-\mathbf{v}_2
\in\mC(\mathbf{\ou},\mathbf{\be},\mathbf{\hu}^*)-\mC(\mathbf{\ou},\mathbf{\be},\mathbf{\hu}^*)$,
where we can verify that
$$\mC(\mathbf{\ou},\mathbf{\be},\mathbf{\hu}^*)-\mC(\mathbf{\ou},\mathbf{\be},\mathbf{\hu}^*)
  =\prod_{k=1}^m\Big(\mC(\ou_k,\be_k,\hu^*_k)-\mC(\ou_k,\be_k,\hu^*_k)\Big).$$
By the polyhedricity of $\mU_{ad}(\mathbf{\be})$ and by
\eqref{CndPlyhrc} we can find sequences
$\mathbf{v}_{1n}\to\mathbf{v}_1$, $\mathbf{v}_{2n}\to\mathbf{v}_2$
and $t_{1n}\downarrow0$, $t_{2n}\downarrow0$ such that
$\mathbf{\ou}+t_{1n}\mathbf{v}_{1n}\in\mU_{ad}(\mathbf{\be})$,
$\mathbf{\ou}+t_{2n}\mathbf{v}_{2n}\in\mU_{ad}(\mathbf{\be})$ and
$\mathbf{v}_{1n},\mathbf{v}_{2n}\in\{\mathbf{\hu}^*\}^\bot$. We
define $t_n:=\min\{t_{1n},t_{2n}\}$ and deduce from the convexity of
$\mU_{ad}(\mathbf{\be})$ that
$$\begin{cases}
    \mathbf{w}_n:=\mathbf{\ou}+t_n\mathbf{v}_{1n}
    =(1-t_nt_{1n}^{-1})\mathbf{\ou}+t_nt_{1n}^{-1}(\mathbf{\ou}+t_{1n}\mathbf{v}_{1n})\in\mU_{ad}(\mathbf{\be})\\
    \mathbf{u}_n:=\mathbf{\ou}+t_n\mathbf{v}_{2n}
    =(1-t_nt_{2n}^{-1})\mathbf{\ou}+t_nt_{2n}^{-1}(\mathbf{\ou}+t_{2n}\mathbf{v}_{2n})\in\mU_{ad}(\mathbf{\be}).\\
\end{cases}$$
Therefore, by choosing $\mathbf{e}_n=\mathbf{\be}$,
$\mathbf{u}^*_n=\mathbf{\hu}^*$, and
$\mathbf{v}_n=\mathbf{v}_{1n}-\mathbf{v}_{2n}$, we have
$(\mathbf{u}_n,\mathbf{e}_n,\mathbf{u}^*_n)\to(\mathbf{\ou},\mathbf{\be},\mathbf{\hu}^*)$
with $(\mathbf{u}_n,\mathbf{e}_n,\mathbf{u}^*_n)\in\gph\mN$ and
$\mathbf{v}_n\to\mathbf{v}$ with
$$\mathbf{v}_n=\mathbf{v}_{1n}-\mathbf{v}_{2n}=\frac{\mathbf{w}_n-\mathbf{u}_n}{t_n}
  \in\cone\big(\mU_{ad}(\mathbf{e}_n)-\mathbf{u}_n\big)\cap\{\mathbf{u}^*_n\}^\bot
  \subset\mC(\mathbf{u}_n,\mathbf{e}_n,\mathbf{u}^*_n).$$
This yields
$\mathbf{v}\in\mC_s(\mathbf{\ou},\mathbf{\be},\mathbf{\hu}^*)$ due
to \eqref{NsLmSTtpA}. $\hfill\Box$
\end{proof}

\begin{Theorem}
For the setting of Theorem~\ref{ThmNcSfCd}, the following assertions
are valid:
\begin{itemize}
\item[{\rm(i)}] If the positive definiteness condition
\begin{equation}\label{CnSPsDfCd}
    \langle F'_\mathbf{u}(\mathbf{\ou},\mathbf{\be})\mathbf{v},\mathbf{v}\rangle
    =\sum_{k=1}^m\sum_{j=1}^m\nabla^2_{u_ku_j}\mJ_k(\ou_1,\ldots,\ou_m,\be_Y,\be_k)v_kv_j>0
\end{equation}
holds for all $\mathbf{v}=(v_1,\ldots,v_m)\in\mathbf{L}^2(\Omega)$
satisfying
\begin{equation}\label{SPsDfCnd}
    0\neq\mathbf{v}\in\prod_{k=1}^m\Big(\cl\big[T(\ou_k;\mU^k_{ad}(\be_k))-T(\ou_k;\mU^k_{ad}(\be_k))\big]
    \cap\{\hu^*_k\}^\bot\Big),
\end{equation}
then $\mathbf{\ou}$ is a fully stable solution to the PVI
\eqref{NshPVInq}.
\item[{\rm(ii)}] If $\mathbf{\ou}$ is a fully stable solution to the PVI
\eqref{NshPVInq} and that $\mU_{ad}(\mathbf{e})$ is polyhedric
around $\mathbf{\be}$, then the condition \eqref{CnSPsDfCd} holds
for all $\mathbf{v}=(v_1,\ldots,v_m)\in\mathbf{L}^2(\Omega)$
satisfying
\begin{equation}\label{CnWPsDfCd}
    0\neq\mathbf{v}\in
    \prod_{k=1}^m\Big(\mC(\ou_k,\be_k,\hu^*_k)-\mC(\ou_k,\be_k,\hu^*_k)\Big),
\end{equation}
where $\mC(u_k,e_k,u^*_k)=T(u_k;\mU^k_{ad}(e_k))\cap\{u^*_k\}^\bot$
for every $k=1,\ldots,m$.
\end{itemize}
\end{Theorem}
\begin{proof}
It follows from Theorems~\ref{ThmNcSfCd}, \ref{ThmUpLwCn} and
Lemma~\ref{LmQLgrFm}. $\hfill\Box$
\end{proof}

\subsection{Case for $\mU_{ad}(\mathbf{e})$ of box constraint type}

It is worthy mentioning that the structure of the admissible control
set $\mU_{ad}(\mathbf{e})$ defined in \eqref{NshAdESet} via
\eqref{NshAdSetK} is standard and it has many nice properties.
Therefore, it is very frequently appearing in the optimal control
theory and applications. Our stability results established in the
previous subsection can be refined for the specific case of the
admissible control set. In this subsection, we will establish
explicit characterizations of full stability for variational Nash
equilibriums to the system \eqref{KthPVI} via the parametric system
\eqref{PerKthPVI} with respect to $\mU_{ad}(\mathbf{e})$ given by
\eqref{NshAdESet} and \eqref{NshAdSetK}.

\begin{Theorem}\label{ThmCmpCne}
Let
$(\mathbf{\ou},\mathbf{\be},\mathbf{\ou}^*)\in\mU_{ad}(\mathbf{\be})\times\mathbf{E}\times\mathbf{L}^2(\Omega)$
such that
$\mathbf{\hu}^*:=\mathbf{\ou}^*-F(\mathbf{\ou},\mathbf{\be})\in\mN(\mathbf{\ou},\mathbf{\be})$,
where $\mU_{ad}(\mathbf{\be})$ is defined by \eqref{NshAdESet} and
\eqref{NshAdSetK}. Assume that the assumptions
{\rm\textbf{(A1)}--\textbf{(A4)}} hold. Then, we have
\begin{equation}\label{ExFrCone}
    \mC_{w^*}(\mathbf{\ou},\mathbf{\be},\mathbf{\hu}^*)
    =\mC_s(\mathbf{\ou},\mathbf{\be},\mathbf{\hu}^*)
    =\prod_{k=1}^m\big\{v\in L^2(\Omega)\bst v(x)\hu^*_k(x)=0~\mbox{for a.a.}~x\in\Omega\big\}.
\end{equation}
\end{Theorem}
\begin{proof}
By our assumptions, for every $k\in\{1,\ldots,m\}$, the admissible
control set $\mU^k_{ad}(e_k)$ in \eqref{NshAdSetK} is convex and
polyhedric for $e_k\in E_k$; see, e.g.,
\cite[Remark~3.5]{QuiDWa19SICON}. Arguing similarly to the proof of
\cite[Lemma~4.5]{QuiDWa19SICON}, we deduce for that
$$\mC_{w^*}(\ou_k,\be_k,\hu^*_k)=\mC_s(\ou_k,\be_k,\hu^*_k)
  =\big\{v\in L^2(\Omega)\bst v(x)\hu^*_k(x)=0~\mbox{for a.a.}~x\in\Omega\big\}.$$
Combining this with \eqref{NshWeStCn} we obtain \eqref{ExFrCone}.
$\hfill\Box$
\end{proof}

\medskip
The forthcoming theorem provides us with an explicit
characterization of full stability for solutions (variational Nash
equilibriums) to the PVI \eqref{NshPVInq}.

\begin{Theorem}\label{ThmChrPVNE}
Assume that all the assumptions of Theorem~\ref{ThmNcSfCd} hold,
where $\mU_{ad}(\mathbf{e})$ is given by \eqref{NshAdESet} and
\eqref{NshAdSetK} for $\mathbf{e}\in\mathbf{E}$. Then $\mathbf{\ou}$
is a fully stable solution to the PVI \eqref{NshPVInq} if and only
if the positive definiteness condition \eqref{CnSPsDfCd} holds for
all $\mathbf{v}=(v_1,\ldots,v_m)\in\mathbf{L}^2(\Omega)$ satisfying
\begin{equation}\label{EquPsDfCd}
    0\neq\mathbf{v}\in
    \prod_{k=1}^m\big\{v\in L^2(\Omega)\bst v(x)\hu^*_k(x)=0~\mbox{for a.a.}~x\in\Omega\big\}.
\end{equation}
\end{Theorem}
\begin{proof}
It follows from Theorems~\ref{ThmNcSfCd}, \ref{ThmCmpCne} and
Lemma~\ref{LmQLgrFm}. $\hfill\Box$
\end{proof}

\medskip
It is interesting to know that there is a relationship between
variational Nash equilibriums that are fully stable under
perturbations and local Nash equilibriums in the classical sense.
From Theorem~\ref{ThmChrPVNE} we obtain the following result on the
aforementioned relationship.

\begin{Theorem}\label{ThmFSvFMzr}
Let $\mathbf{\ou}\in\mS(\mathbf{\ou}^*,\mathbf{\be})$ and put
$\mathbf{\hu}^*:=\mathbf{\ou}^*-F(\mathbf{\ou},\mathbf{\be})$, where
$\mU_{ad}(\cdot)$ is defined by \eqref{NshAdESet} and
\eqref{NshAdSetK}. Assume that the assumptions
{\rm\textbf{(A1)}--\textbf{(A4)}} hold. If $\mathbf{\ou}$ is a fully
stable solution to the PVI \eqref{NshPVInq}, then $\mathbf{\ou}$ is
a local Nash equilibrium associated to the parametric optimal
control problems
\begin{equation}\label{TlPrKthMzr}
    \mbox{\rm Minimize}\quad
    \mJ_k(u_k,u_{-k},e_Y,e_k)-(u^*_k,u_k)_{L^2(\Omega)}~~\mbox{\rm subject to}~~u_k\in\mU^k_{ad}(e_k)
\end{equation}
with respect to $(\mathbf{\ou}^*,\mathbf{\be})$, where the
functional $\mJ_k(u_k,u_{-k},e_Y,e_k)$ is defined in
\eqref{BaPaCFctn}.
\end{Theorem}
\begin{proof}
Applying Theorem~\ref{ThmChrPVNE}, we infer that the positive
definiteness condition \eqref{CnSPsDfCd} holds for all
$\mathbf{v}=(v_1,\ldots,v_m)\in\mathbf{L}^2(\Omega)$ satisfying
\eqref{EquPsDfCd}. This implies that for every $k\in\{1,\ldots,m\}$
the condition
\begin{equation}\label{}
    \nabla^2_{u_ku_k}\mJ_k(\ou_k,\ou_{-k},\be_Y,\be_k)v_kv_k>0
\end{equation}
holds for all $v_k\in L^2(\Omega)$ with $v_k\neq0$ and
$v_k(x)\hu^*_k(x)=0$ for a.a. $x\in\Omega$. Combining this with
\cite[Theorem~4.8]{QuiDWa19SICON} we deduce that $\ou_k$ is a
(Lipschitzian and H\"{o}lderian) fully stable local minimizer of the
parametric control problem \eqref{TlPrKthMzr} with respect to
$(\ou^*_k,\be_Y,\be_{k,J},\be_{k,\alpha},\be_{k,\beta})$ for every
$k=1,\ldots,m$; see definitions of Lipschitzian and H\"{o}lderian
fully stable local minimizers in \cite{MorNgh14SIOPT,QuiDWa19SICON}.
This implies that $\mathbf{\ou}$ is a local Nash equilibrium
associated to the parametric control problem \eqref{TlPrKthMzr} with
respect to $(\mathbf{\ou}^*,\mathbf{\be})$. $\hfill\Box$
\end{proof}

\begin{Remark}\rm
Theorems~\ref{ThmChrPVNE} and \ref{ThmFSvFMzr} can be viewed as
generalizations of \cite[Theorem~2.2]{Cas12SICON} (see also the
results in \cite{CasTro12SIOPT}) to the perturbed cases.
\end{Remark}

\section{Concluding remarks}

In this paper, we have provided some new results on the existence of
variational/classical Nash equilibriums to the equilibrium problem
associated to the nonconvex/convex optimal control problems governed
by semilinear elliptic partial differential equations. In addition,
we have established a necessary condition and a sufficient condition
(resp., an explicit characterization) of full stability for
variational Nash equilibrium to the parametric equilibrium problem
under full perturbations for general nonempty bounded closed convex
component admissible control sets (resp., for component admissible
control sets of box constraint type). Furthermore, for the case
where the component admissible control sets of box constraint type,
we have also proved that variational Nash equilibriums and local
Nash equilibriums to the parametric equilibrium problem are
equivalent provided that the variational Nash equilibriums are fully
stable.

\end{document}